\documentclass{amsart}
\usepackage[english]{babel}
\usepackage{verbatim}
\usepackage{amssymb, amsfonts, amsthm, amsmath, amscd, mathrsfs, mathtools, bbm}
\usepackage{graphicx}
\usepackage{enumerate}
\usepackage{tikz,caption,subcaption}
\usetikzlibrary{decorations.markings}
\usepackage{pgfplots}
\usetikzlibrary{arrows,calc}

\def\R{{\mathbb R}}
\def\N{{\mathbb N}}
\def\C{{\mathbb C}}
\def\Z{{\mathbb Z}}
\def\1{{1\!\!1}}
\def\E{{\mathbb E}}

\title{ Stationary analysis of the shortest queue problem}

\author[S.~Dester]{Plinio S.~Dester}
\address[Plinio S.~Dester]{INRIA Paris, 2 rue Simone Iff,  CS 42112, 75589 Paris Cedex 12, France}
\email{plinio.santini-dester@polytechnique.edu}

\author[Fricker]{Christine Fricker}
\address[Christine Fricker]{INRIA Paris, 2 rue Simone Iff,  CS 42112, 75589 Paris Cedex 12, France}
\email{christine.fricker@inria.fr}

\author[Tibi]{Danielle Tibi} 
\address[Danielle Tibi]{LPMA - Universit\'e Paris Diderot, b\^atiment Sophie Germain,  case courrier 7012,  8 place Aur\'elie Nemours, 75205 Paris cedex 13, France}
\email{tibi@math.univ-paris-diderot.fr}

\def\R{{\mathbb R}}
\def\N{{\mathbb N}}
\def\C{{\mathbb C}}
\def\Z{{\mathbb Z}}
\def\1{{1\!\!1}}
\def\E{{\mathbb E}}

\newtheorem{theorem}{Theorem}
\newtheorem{proposition}{Proposition}
\newtheorem{corollary}{Corollary}
\newtheorem{lemma}{Lemma}
\newtheorem{remark}{Remark}


\begin{document}



\footnote{\hspace{-4mm}\noindent\fbox{\parbox{\linewidth-2\fboxrule-2\fboxsep}{This is an electronic, extended version of the article published in 
      Queueing Systems, 87(3), 211-243 (2017). This reprint differs from original in pagination and typographic detail. It includes supplementary material, mainly,  the proof of Proposition 2, a sketch of proof of the result for the asymmetric model, and a simple alternative proof of a result by J.~W.~Cohen. In contrast, the Section entitled  {\em Application to large-scale analysis} is  not included.  }}}

      
\maketitle

\begin{abstract}
A simple analytical solution is proposed for the stationary loss  system of two parallel queues with finite capacity $K$, in which new customers join the shortest queue, or one of the two with equal probability if their lengths are equal. The arrival process is Poisson, service times at each queue have exponential distributions with the same parameter, and both queues have equal capacity. Using standard generating function arguments, a simple expression for the blocking probability is derived, which  as far as we know is original. 
Using coupling arguments and explicit formulas, comparisons with  related loss systems are then provided. Bounds  are similarly obtained for the average total number of customers, with the stationary distribution explicitly determined  on $\{K, \dots, 2K \}$, and elsewhere upper bounded.  
   Furthermore, from the balance equations,  all stationary probabilities are obtained as explicit  combinations of their values at states $(0,k)$ for $0 \le k \le K$.
  These  expressions extend to  the infinite capacity and   asymmetric cases, i.e.,  when the queues have different service rates.   For the initial symmetric  finite capacity model, the  stationary probabilities of states $(0,k)$ can be obtained recursively from the blocking probability. In the other cases, they are implicitly determined through a  functional equation that characterizes their generating function.  The whole approach shows that the stationary distribution of the infinite capacity symmetric   process   is the limit of the corresponding finite capacity distributions. 
For the infinite capacity symmetric model, we provide an elementary proof  of a result by Cohen which gives the solution of the functional equation in terms of    an  infinite product with explicit zeroes and poles.

\end{abstract}

\section{Introduction}\label{sec1} The \emph{join--the-shortest-queue} (JSQ) policy is  used for load-balancing purposes in stochastic networks. 
Yet for  systems involving such a mechanism, an  exact stationary  analysis  is far from trivial. The simplest model with two queues, known as  \emph{two queues in parallel} or \emph{join-the-shortest-queue model},  has itself given rise to an abundant literature. Most of it,  from the first paper by Haight \cite{haight1958two} to the complete solution by Flatto and MacKean \cite{Flatto-1} and by Cohen \cite{cohen1995two}, is devoted to  the  infinite capacity  model. The latter indeed raises  interesting issues of complex analysis.  The approach initiated by Kingman \cite{Kingman-1} is to characterize the invariant distribution through its bivariate generating function $F(x,y)$, which can be expressed in terms of the univariate functions $F(x,0)$ and $F(0,y)$.  Those are characterized via functional equations. Using complex variable arguments, \cite{Kingman-1}  proves that both functions have a meromorphic continuation in the complex plane, and then determines their  poles and residues. As a by-product, using partial fraction expansion, the stationary probabilities of states $(0,k)$ and $(k,k)$ for $k \in \N$ are  derived as infinite sums, involving these poles and residues.   Asymptotics are then derived for the  stationary probabilities  (in the limit of large states) and  for the  waiting time distribution at stationarity (in the heavy-traffic limit). In the same spirit,  \cite{Flatto-1} and \cite{cohen1995two} further obtain  different expressions for $F(x,0)$ and $F(0,y)$.  In  \cite{Flatto-1}, those  are formulated for $(x,y)$ on a particular Riemann surface,  using a uniformizing variable. Asymptotics are further provided  for the stationary probabilities,  improving those of \cite{Kingman-1}, and for the mean number of customers.   In \cite{cohen1995two}, both generating  functions are  represented as  infinite products, derived from their zeroes and poles  by using the Weierstrass factorization theorem. All stationary probabilities are then expressed as infinite sums involving the  poles and residues of $F(x,0)$ and $F(0,y)$,   generalizing the expressions in \cite{Kingman-1}.  Next, \cite{cohen1998analysis} extends Kingman's results to the model where the queues have two different service rates, determining the poles and residues of the meromorphic functions of interest.   The \emph {compensation approach}  (see \cite{adan1990analysissym,adan1991analysisasym,adan1993compensation}) produces explicit expressions for the stationary probabilities as infinite series of product forms. It more generally applies to a whole class of two-dimensional nearest-neighbor random walks. The structure of this series representation is suited for approximations  and numerical evaluation. For JSQ, it is shown in \cite{cohen1996symmetrical} that the series derived via the compensation method coincide with those given by the analytical approach.  
In \cite{fayolle1979two}, Fayolle and Iasnogorodski  characterize  the meromorphic continuations of $F(x,0)$ and $F(0,y)$
through a Riemann-Hilbert boundary value problem.   See also \cite[Chapter~10]{fayolle2017random} and the references therein. From this description, Kurkova and Suhov  \cite{Kurkova-1} obtain asymptotics for the stationary probabilities in a model where customers join the shortest queue only with some given probability.  For the same model, 
 the decay of stationary probabilities is analyzed via the matrix analytic method  by Li et al.~\cite{li2007geometric} 
  (see also references therein),   through the Markov additive approach by Foley and Mc Donald \cite{foley2001join}, 
   in heavy traffic \cite{turner2000join} and via large deviations  \cite{turner2000large} by Turner.
 In \cite{knessl1986two}, Knessl et al.~derive heuristics 
 from the balance equations.
Other heuristic approaches have been developed, aiming at  numerical results, like the \emph {power series algorithm} (see  Hooghiemstra et al.~\cite{Hooghiemstra1988Power}, Blanc \cite {blanc1992power}), that are quite accurate, but fail to have a theoretical justification. Similar computational procedures are proposed by  Rao and Posner \cite{rao1987algorithmic}.

Bounds on the stationary probabilities and mean total number of customers are derived by Halfin \cite{Halfin-1}. Van Houtum et al.~\cite{van1998bounds} obtain bounds via stochastic comparison of cost structures. 
 For the  JSQ model of $n$ queues, Winston \cite{winston1977optimality} proves that, among policies that immediately assign customers to a queue, JSQ is optimal for Poisson arrivals and exponential service times: It maximizes, in terms of stochastic order, the number of customers served in a given time interval. Weber \cite{weber1978optimal} extends this result to a more general class of service times with non-decreasing hazard rate and a general arrival process.  In \cite{whitt1986deciding}, Whitt  exhibits counterexamples for general service time distributions. For $n$ queues, large deviation results are obtained by Ridder and Schwartz \cite{ridder2005large} and Puhalskii and Vladimirov \cite{puhalskii2007large}. When the number of queues $n$ scales with the global arrival rate $\lambda_n$,  Eschenfeldt and Gamarnik \cite{eschenfeldt2015join} study the behavior in the Halfin-Whitt regime, i.e., $(1-\lambda_n)\sqrt{n}\rightarrow \beta$. See also references on the heavy-traffic regime therein.

 Variants have also been investigated, like jockeying, where customers can change queue during service (see, for example, \cite{adan1991analysis}), or the shortest-queue-first model, where  they join the queue with  minimum workload \cite{guillemin2014stationary}, or serve-the-longest-queue model, SLQ, where one server moves between two queues, joining the longest one when completing  a service \cite{flatto1989longer}.

The finite capacity version of the shortest queue problem was first investigated by Conolly \cite{Conolly-1} and then by Tarabia \cite{Tarabia-1} who uses matrix analysis methods. Both analyze  the steady-state probabilities with a view to giving numerical results.  Based on singular perturbation expansions  within the balance equations,  Knessl  and Yao  \cite{knessl2011finite}  derive asymptotics of the stationary probabilities as the capacity $K$ gets large, rescaling the state space into a size-one square. They obtain different behaviors according to different  regions of the  state space, and their results are, on the whole, numerically accurate even for small values of $K$.

The prime objective of the present paper is to derive simple expressions for the steady-state probabilities of the JSQ system with finite capacity $K$. 
Since some of the results remain valid for infinite capacities and  non-equal service rates, the scope has been extended to include those cases. This work is  motivated by the issue of approximating large real systems with a local choice strategy, such as vehicle-sharing networks,  as considered in \cite{Dester2017Questa}.
It is classical that for a system of $N$ identical one-server queues with Poisson arrivals, if customers join the shortest queue among two queues chosen at random, then   in the mean-field limiting regime of  large $N$,  the   stationary  number of customers per queue  falls from exponential to double-exponential decrease (\cite{mitzenmacher1997analysis,vvedenskaya1996queueing} and others). This demonstrates the so-called \emph{power of two choice}.   Similar balancing policies could be used to improve resource availability in bike-sharing systems. In this regard, a system where  users  return bikes to the less loaded of two stations chosen at random among all the stations has been studied in Fricker and Gast \cite{fricker2014incentives}. Obviously, only a local choice  policy  makes sense in practice, but the underlying dynamics are analytically intractable.
To circumvent the difficulty, these can be handled by clustering the network into groups of stations which can collaborate, say, for simplicity groups of two. Within this framework, mean-field limits involve the stationary mean number of customers of the typical object, that is, the JSQ model studied here.

The shortest queue problem is first considered for two one-server queues with  finite capacity $K$. A simple exact expression for the blocking probability, together with the probability that a particular queue is empty, is obtained by adapting Kingman's generating function method, here  using the functional equation for $F(0,y)$  at some specially chosen 
values.  This result  extends to the case of an additional constraint on the total number of customers admitted in the system.  The blocking probability  is  next quantitatively analyzed  and compared, on the one hand, to the loss probability  of two independent $M/M/1/K$ queues, and on the other hand, to that of one unique two-server queue with double capacity, i.e., an $M/M/2/2K$ queue.  The comparison is made through stochastic ordering 
 and  evaluating the uniform distance between blocking probabilities. Asymptotics are also derived in the different ranges of values of the parameters. The stationary number of customers in the system is then analyzed.  Part of its distribution is explicit, and its mean is given accurate bounds, notably involving the explicit blocking probability.  
 Independently of the blocking probability result, the balance equations are next solved via a recursive procedure involving discrete convolution products. All stationary probabilities are then obtained as explicit combinations of their particular  values at states  $(0,k)$ for $k=0, \dots ,K$. These expressions straightforwardly extend to the models with either  infinite capacity or  different service rates, thus providing a unified statement for  the finite/infinite or symmetric/asymmetric models. In the symmetric  case, this  makes it possible to prove  that the stationary distribution of the infinite capacity  process, when ergodic, is the limit of that of the finite capacity $K$ model, as $K$ goes to infinity. 
   The  unknown stationary probabilities  of states  $(0,k)$ --and $(k,0)$ for the asymmetrical case-- are characterized through functional equations. This alternative to the classical description of the stationary distribution through its  generating function has the advantage that only  $F(0,y)$ is involved. For the symmetric infinite capacity case, a short and  elementary proof of the infinite product representation of   $F(0,y)$ by Cohen \cite{cohen1995two} 
  is given.  This proof  essentially avoids the use of complex variable arguments by finding an obvious solution of the functional equation  and  then proving uniqueness under a condition of analyticity in some disk. 
   
   Let us mention that  the whole analysis can be adapted to the dual  \emph{serve-the-longest-queue} (SLQ) model. More precisely, when the  capacity $K$ of the queues is finite, the two models are derived from each other by exchanging occupied and vacant space in each queue. As a result, the stationary blocking probability of the JSQ model is equal to the stationary probability of idleness of the server  in the corresponding  SLQ model.

    The  JSQ  model is a particular example of a non-homogeneous \emph{quasi-birth-and-death} process. Our expressions for the stationary probabilities as linear combinations of those  of ``level zero'' amount to identifying 
 a set of rate matrices.  Those have  entries of both signs, which excludes probabilistic interpretations. 
  The method 
consists in solving separately a homogeneous subsystem of the balance equations. This is made possible, through convolution products, due to the absence of upward transitions inside some contour. This technique can extend to other models with no upward --or no downward-- one-step transitions, except on a boundary, such as those considered in \cite{van2009quasi}.
It constitutes  a simple alternative to the \emph{lattice path counting} procedure.   Note that for JSQ, the specific  positions of the upward jumps mean that the reordered queue-length process is \emph {successively lumpable}, as defined in  \cite{katehakis2012successive}. 
   Indeed, the larger component can increase to some level  $n$ only through the \emph{entrance state}  $(n-1,n)$. Yet, our analysis is entirely different from the successive lumping algorithm, since   we first solve the set of balance equations that \emph {do not} involve the one-step transitions toward the entrance states, while  those are involved  in building  accessory processes 
   at each stage of the successive lumping algorithm. 
   Moreover, it seems that explicit formulas for the stationary probabilities are out of reach for the latter approach.
         Note also that for JSQ with infinite capacity, there is no initial stage to begin the 
         algorithm. More generally, rate matrices for non-homogeneous QBD are usually characterized through a recursive scheme that is only solved numerically,  requiring space truncation arguments when the  set of levels is infinite (see \cite{katehakis2015DES,latouche}).  The present direct approach to the invariant measure of JSQ is thus original. It also differs radically from the compensation method. The latter indeed derives successive approximations  by a series of product form terms, while we obtain  \emph {finite} sums of \emph {non-product} terms.  Those are explicit, except for coefficients given by the stationary probabilities along the axes, that are characterized through either recursion, or a functional equation for their generating function. Regarding JSQ, as far as we know from the  literature, no previous work has thus addressed its different variants together with the same approach.


 
 The paper is organized as follows. Section 2 analyzes the symmetric finite capacity model, first focusing on the blocking probability, then on the mean total number of customers,  and next characterizing the whole invariant distribution. Section 3 deals with the symmetric infinite capacity case,    to which the last part of Section 2 is extended. Weak convergence of the finite capacity stationary distribution, as $K$ goes to infinity, is established under ergodicity of the infinite capacity process. The alternative proof to the result of  \cite{cohen1995two} is then provided. Finally, Section 4 states  similar  characterizations of the invariant distribution for the asymmetric, finite or infinite capacity models. 
 
 The symmetric model, which  is the one mainly considered, consists  of two one-server queues, each having capacity $K$ that may be finite or infinite. Service times at both queues are exponentially distributed with the same parameter that can be set equal to $1$   without loss of generality. Customers arrive according to a Poisson process with parameter $2 \rho$ and join the shortest queue, or  either queue  with probability $1/2$  if both are equal. If $K < \infty$, then when both files have length $K$, new  customers are rejected and definitively lost.

\section{The finite capacity model }\label{sec2}
Here, the queues have the same finite capacity $K$ ($K \in \N$). Therefore, when both queues have $K$ customers, any new arriving customer is definitively rejected from the system.  Denoting by $L_i(t)$, for $i=1,2$ and $t \ge 0$, the number of customers at queue  $i$ at time $t$,  the queue-length  process $(L_1(t), L_2(t))_{t \ge 0}$ is Markov with state space $\mathcal S _K = \{0, \dots , K \} ^2$. Its $Q$-matrix, denoted $Q_K$,  is characterized by the following jumps and rates, where $e_1 $ and $e _2$ are the units vectors in $\R^2$ :
\begin{itemize}
\item  for $v=(k,k)$,
\[ \begin{array}{l} Q_K(v,v+e_1) ~=~ Q_K(v, v+e_2) ~=~ \rho \, \1_{\{k<K\}} \\
Q_K(v,v-e_1) ~=~ Q_K(v, v-e_2) ~=~  \1_{\{k>0 \}} \end{array} \]
\item for  $v=(j,k)$ with $ j <k $, and $v' = (k,j)$,
 \[  \begin{array}{l}  Q_K(v,v+e_1) ~=~ 2 \rho ~=~ Q_K(v', v'+e_2) \\
 Q_K(v,v-e_1) ~=~  \1_{\{ j >0 \}}  ~=~ Q_K(v', v'-e_2) \\
 Q_K(v, v-e_2) ~=~  1 ~=~ Q_K(v',v'-e_1) \end{array}, \]
 \end{itemize}
  as represented in Figure~\ref{transitions} --where by symmetry,  the lower half-space is omitted.
  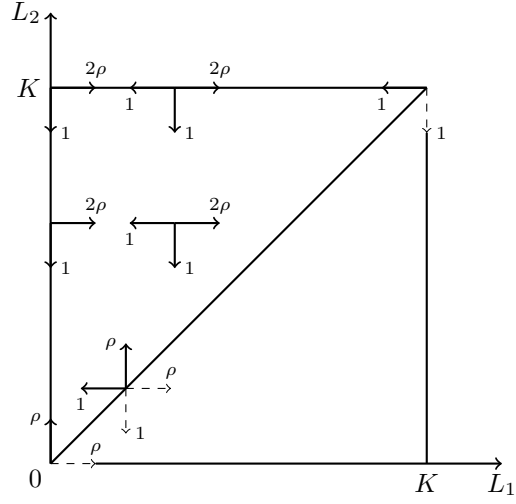
\begin{figure}[ht]
		\centering
		\begin{tikzpicture}
			\draw[->, thick]
			 (0.6,0) -- (6,0) node[below] {$L_1$}
			 ;
			\draw[->, thick]
			 (0,0) -- (0,6) node[left] {$L_2$};
	
			\node at (-.2,-.2) {$0$};

			\draw[-, thick] (0,0) -- (5,5);
			
			\draw[ thick]
			 (0,5) node[left] {$K$}-- (5,5) ;
			\draw[ thick]
			 (5,0) node[below] {$K$}-- (5,4.4) ;

			\draw[->, thick]
				(1,1) -- ++	(0,.6)	node[left]	{\scriptsize{$\rho$}};
			\draw[->, thick]
				(1.,1.) -- ++	(-.6,0)	node[below]	{\scriptsize{$1$}};
			\draw[->, dashed]
				(1.,1.) -- ++	(0,-.6)	node[right]	{\scriptsize{$1$}};
			\draw[->, dashed]
				(1.,1.) -- ++	(.6,0)	node[above]	{\scriptsize{$\rho$}};

			\draw[->, thick]
				(0,0) -- ++	(0,.6)	node[left]	{\scriptsize{$\rho$}};
		
			\draw[->, dashed]
				(0,0) -- ++	(.6,0)	node[above]	{\scriptsize{$\rho$}};

			
			\draw[->, thick]
				(5,5) -- ++	(-.6,0)	node[below]	{\scriptsize{$1$}};
			\draw[->, dashed]
				(5,5) -- ++	(0,-.6)	node[right]	{\scriptsize{$1$}};
			
		
			\draw[->, thick]
				(0,5) -- ++	(0,-.6)	node[right]	{\scriptsize{$1$}};
			\draw[->, thick]
				(0,5) -- ++	(0.6,0)	node[above]	{\scriptsize{$2\rho$}};

			
			\draw[->, thick]
				(1.65,3.2) -- ++	(-.6,0)	node[below]	{\scriptsize{$1$}};
			\draw[->, thick]
				(1.65,3.2) -- ++	(0,-.6)	node[right]	{\scriptsize{$1$}};
			\draw[->, thick]
				(1.65,3.2) -- ++	(0.6,0)	node[above]	{\scriptsize{$2\rho$}};

			\draw[->, thick]
				(1.65,5) -- ++	(-.6,0)	node[below]	{\scriptsize{$1$}};
			\draw[->, thick]
				(1.65,5) -- ++	(0,-.6)	node[right]	{\scriptsize{$1$}};
			\draw[->, thick]
				(1.65,5) -- ++	(0.6,0)	node[above]	{\scriptsize{$2\rho$}};

		
			\draw[->, thick]
				(0,3.2) -- ++	(0,-.6)	node[right]	{\scriptsize{$1$}};
			\draw[->, thick]
				(0,3.2) -- ++	(0.6,0)	node[above]	{\scriptsize{$2\rho$}};

		\end{tikzpicture}
		\caption{Transition rates of Markov process $(L_1,L_2)$}
\label{transitions}
	\end{figure}

 The process is clearly irreducible,  thus admitting a unique invariant probability distribution $\pi _K$, which is the object of interest in this section. By symmetry, 
 \[\pi_K(j,k) ~=~ \pi _K(k,j) \quad \text{ for } (j,k) \in \mathcal S_K,\] 
and the balance equations 
  are reduced to
 \begin{equation} \label{balance}    \begin{cases}  
   \big (\1 _{\{k>0\} } + \rho \, \1 _{ \{k<K \}}\big ) \,    \pi_K(k,k) ~=~     2\rho \, \1 _{ \{k>0 \}} \,   \pi_K(k-1,k) \\
 \hspace{35mm}  ~+~  \1 _{\{k<K \}} \,  \pi_K(k,k+1),\quad k= 0, \dots ,K, \medskip  \\  
 \left (\1 _{\{j>0\}} + 1 + 2\rho \right )    \pi_K(j,k) =     2\rho \, \1 _{\{j>0\}} \,   \pi_K(j-1,k)  + \pi_K(j+1,k)   \\ 
  \hspace{2mm}  ~+~  \1 _{\{k<K\}} \,  \pi_K(j,k+1)  ~+~ \rho \,  \1 _{\{k= j+1\}} \pi _K (j,j),\quad 0  \le j <k \le K.
      \end{cases}
        \end{equation}

   \subsection{Stationary blocking probability.}
   
   The classical approach to the invariant distribution for infinite $K$ (see \cite{cohen1995two,Flatto-1,Kingman-1}) is through the bivariate generating function  of the stationary queue-length vector. The same  method will  here be used for $K < \infty$,  leading to the determination of the stationary blocking probability $\pi_K(K,K)$.    
   Define for $x,y \in \C$,
   \begin{align*} F_K(x,y)~\stackrel{def}{=} ~\E \left (x^{L_1} y^{L_2-L_1} \1 _{\{L_1 \le L_2\}} \right ) 
   ~= \sum _{0 \le j \le k \le K} \pi_K(j,k) \, x^j y ^{k-j}  \, , 
   \end{align*} 
    where $(L_1,L_2)$ denotes the queue-length vector at stationarity.

   Proceeding as in \cite{Kingman-1}, one can convert the balance equations into a functional equation that characterizes $F_K$.   Note that, contrary to the infinite capacity case, $F_K$ is here defined for all complex values of $x$ and $y$. Apart from this notable difference,  the computation is similar to \cite{Kingman-1} and leads to the following relation:
   \begin{multline} \label{F_K}  \Big ( y^2 - 2(1+ \rho) xy + (1+2 \rho x) x  \Big) \, F_K(x,y)  ~=~ y(y-x) \, A_K(y)  \\
   ~-~ \Big (  \rho y^2+ (1+ \rho) y - 1-2 \rho x  \Big)  x B_K(x)  ~+~ \rho \,  x^{K+1} y(y-1) \, \pi_K(K,K), 
   \end{multline}
   where $A_K$ and $B_K$ are  univariate generating functions given by
   \[ A_K(y) = F_K(0,y) = \sum _{k=0}^K \pi_K(0,k) \, y^k \quad \text{ and } \quad B_K(x)= F_K(x,0)  =  \sum _{j=0}^K \pi_K(j,j) \, x^j. \]
   [Relation \eqref{F_K} is identical to the one derived by Kingman for $K= \infty$, except for the additional  term  here involving the blocking probability $\pi _K(K,K)$.]
   
   $F_K$ is thus determined as a function of $A_K$, $B_K$ and $\pi _K(K,K)$. We now use the standard argument  that whenever the left-hand side vanishes, the right-hand side must also vanish. In particular, 
      \[  y(y-x)  A_K(y) ~-~ \Big (  \rho y^2+ (1+ \rho) y - 1-2 \rho x  \Big)  x B_K(x)  ~+~ \rho   x^{K+1} y(y-1)  \pi_K(K,K)=0 
\]
 for all $x,y \in \C$ such that  $ p_x(y)=0$, 
where $p_x$ is defined  for $x \in \C$ by
 \begin{equation} \label{p_x}  p_x(Y) ~\stackrel{def}{=}~Y^2 - 2(1+ \rho) xY + (1+2 \rho x) x .
 \end{equation}
 
  Now, for fixed $x$, denote by $y$ and $z$ the two (possibly equal) roots  of   $ p_x$. Eliminating $B_K(x)$ within both relations  obtained for couples $(x,y)$ and $(x,z)$   yields
\begin{multline*}  \Big  ( \rho z^2+ (1+ \rho) z - 1-2 \rho x \Big) \Big ( y(y-x) \, A_K(y)   + \rho \,  x^{K+1} y(y-1) \, \pi_K(K,K) \Big ) =  \\
  \Big (\rho y^2+ (1+ \rho) y - 1-2 \rho x  \Big )  \Big ( z(z-x) \, A_K(z)   +\rho \,  x^{K+1} z(z-1) \, \pi_K(K,K) \Big ).
  \end{multline*}
  Using $p_x(y)= 0$, we get  $ \rho y^2+ (1+ \rho) y - 1-2 \rho x = (1+2 \rho x) \big ((1+ \rho )y -(1+ \rho x) \big )$ and the analogue with $z$ in place of $y$, so that for $1 +  2 \rho  x \neq 0$, 
    \begin{multline} \label{A_K_step}
     \Big ((1+ \rho )z -(1+ \rho x)  \Big ) \Big ( y(y-x) \, A_K(y)   + \rho \,  x^{K+1} y(y-1) \, \pi_K(K,K) \Big ) =   \\
       \Big ( (1+ \rho )y -(1+ \rho x) \Big )  \Big ( z(z-x) \, A_K(z)   +\rho \,  x^{K+1} z(z-1) \, \pi_K(K,K) \Big ).
  \end{multline} 
   Next, on the one hand,  we use $y+z= 2(1+ \rho)x$ and again  $p_x(z)=0$ to get
  \[ (y-x)\big ((1+ \rho )z -(1+ \rho x) \big )~=~ \big ((1+ 2\rho )x - z \big )\big ((1+ \rho )z -(1+ \rho x) \big ) ~=~ (x-1) (\rho x-z), \]
and the same with $y$ and $z$ exchanged.  On the other hand, from  relations $yz = (1+2 \rho x)x$ and $y+z= 2(1+ \rho)x$,  we derive 
\[ z(z-1)\big ((1+ \rho )y -(1+ \rho x) \big ) - y(y-1)\big ((1+ \rho )z -(1+ \rho x) \big )~=~ (x-1)(y-z). \] 
Equation \eqref {A_K_step} then yields, for $x \neq  1$ and $1 +  2 \rho  x \neq 0$,
\[y A_K(y) (\rho x-z) -z A_K(z) (\rho x-y) ~=~ \rho \,  x^{K+1}  (y-z) \, \pi_K(K,K),\] 
or equivalently, multiplying by $2(1+ \rho)$ and using again $2(1+ \rho)x = y+z$,
   \[  \big (\rho y-(2+ \rho)z \big )   y  A_K(y) -  \big (  \rho z-(2+ \rho)y \big )    z  A_K(z)   = 2 \rho (1+ \rho) x^{K+1} (y-z)  \, \pi_K(K,K) .\]
   This equation is analogous to the one derived in \cite{Kingman-1} for $K = \infty$, in which case the right-hand side is zero. One can go one step  further, noting that, 
   \[ y  \big (\rho y-(2+ \rho)z \big )  = 2 \rho  x \left  (\rho y-z - \frac{1+ \rho}{ \rho} \right ), \] 
 [using again $p_x(y)=0$, together with $yz = (1+2 \rho x)x$].   We finally get that
 \begin{equation} \label{A_K}
 \phi (y,z) A_K(y) - \phi(z,y) A_K(z) ~=~  (1+ \rho) \,  x^{K}  (y-z) \, \pi_K(K,K)
 \end{equation}
 for all $x \in \C \setminus \{0, 1, -(2 \rho)  ^{-1} \}$, where $y,z$ are the roots of $p_x$ and 
  \begin{equation} \label{phi}    \phi (y,z) ~=~ \rho y-z - \frac{1+ \rho}{ \rho},  \qquad y,z \in \C.
  \end{equation} 
  Now from continuity of the set of roots of a polynomial with respect to its coefficients, equation \eqref{A_K} extends to all complex values of $x$.

    Using  equation \eqref{A_K},  the computation of $\pi_K(K,K)$ will be  made possible thanks to a  particularly nice property. Indeed, for two special  values of $x$, the two associated pairs of roots  have a common element $y$, which is coupled, on one side, with some  $z$ such that $\phi(z,y)=0$, and on the other side with $1$, at which $A_K$ can be independently evaluated.
  
  \begin{theorem} \label{blocking} For $\rho >0$, 
  \[ \pi_K(K,K) =  \left \{ \begin{array} {ll} \displaystyle{\frac {(1- \rho)(2 - \rho )}{ \rho ^{-2K} +(1- \rho)(2 \rho )^{-K} - \rho(2- \rho)}} & \  \text{for } \   \rho \notin \{ 1,2 \}  \medskip \\
  \smallskip
 \left ( 2K + 2^{-K} \right ) ^{-1} &\   \text{for } \  \rho  =1\\
\left ( 2 - (K+2)2^{-(2K+1)} \right ) ^{-1}  &\   \text{for } \  \rho  =2,
\end{array} \right . \]
or equivalently,
\begin{equation} \label{pi(K,K)} \pi _K(K,K) =  \frac {2 \rho ^{2K}} {2\sum _{k=0}^{2K} \rho ^k - \sum _{k=0}^{K-1} (\rho/2) ^k} \cdot  
\end{equation} 
\end{theorem}
  
  \begin{proof}
For  $x= 1/ \rho ^2$,  the  roots of  $p_x$   are $y =1/ \rho  $ and $z=(2+\rho)/\rho ^2$, which satisfy $\phi (z,y) = 0$.   Relation \eqref{A_K} then yields
 \begin{equation} \label {A_K(1/rho)} A_K(1/ \rho)  ~=~\rho ^{-2K} \pi_K(K,K). 
 \end{equation}

  Now, $1/ \rho $ is also a root  of $p_x$ for $x= 1/2\rho $,   the other root  here being $1$.  We then get, from  \eqref{A_K}, 
  \[(2- \rho ) A_K(1) - A_K(1/ \rho)  ~=~ (1- \rho) (2 \rho) ^{-K}  \pi_K(K,K). \] 
  By summing both relations,  we have that
  \begin{equation} \label {firstA_K(1)} (2- \rho ) A_K(1)  ~=~ \left ( \rho ^{-2K} + (1- \rho) (2 \rho) ^{-K} \right ) \pi_K(K,K). 
  \end{equation}
  Another expression of $A_K(1)$ can be obtained, using  the relations 
   \begin{equation} \label{generating}
    2 F_K(x,y) - B_K(x) ~=~ \E \left (x^{\min (L_1,L_2)} y^{\max (L_1, L_2)-\min( L_1, L_2)} \right )   
   \end{equation}
   for  $x,y \in \C$, that result from symmetry, and
\[ F_K(x,x) ~=~ (1+ \rho x) B_K(x)  - \rho \,  x^{K+1} \pi _K(K,K),  \quad x \in \C , 
\]
 together with
 \[ (1-2 \rho x)\,  F_K(x,1)  ~=~  A_K(1) -2 \rho x B_K(x),  \quad x \in \C , \]
 that result from  \eqref{F_K}.  
 In particular, for $x=y =1$, one gets  
  \[ 2 F_K(1,1) - B_K(1) =   1, \] 
   \[F_K(1,1) ~=~ (1+ \rho ) B_K(1)  - \rho \pi _K(K,K) \]
 and
 \[ (1-2 \rho )\,  F_K(1,1)  ~=~  A_K(1) -2 \rho  B_K(1). \]  
 From the last three relations, one easily derives  that 
  \begin{equation} \label {A_K(1)}
  A_K(1) = 1- \rho \,  \big (1 - \pi _K (K,K) \big ), 
  \end{equation} 
 which, together with  relation \eqref{firstA_K(1)},  yields   
   \[ \big ( \rho ^{-2K} +(1- \rho)(2 \rho )^{-K} - \rho(2- \rho) \big ) \,  \pi_K(K,K)  ~=~ (1- \rho)(2 - \rho ) .\]
   Since $\pi _K(K,K)$ is defined for all positive values of  $\rho$, the first factor in the left-hand side must vanish only at $\rho =1$ or $2$.  
   The expression of $\pi_K(K,K)$ is thus determined for $\rho \notin  \{ 1,2 \}$. Values at $1$ and $2$ are obtained by taking limits, using continuity of $\pi_K(K,K)$ with respect to $\rho >0$. Indeed, $ ( \pi_K(j,k), 0 \le j,k \le K)$ is continuous with respect to $\rho >0$, as the unique solution of a system  of linear equations with continuous coefficients, which consists of the balance equations together with $ \sum _{j,k} \pi_K(j,k)    =1$. 
   
   The alternate expression  of $\pi_K(K,K)$, valid for all  $\rho$, is  derived  by writing
\begin{align*}
  \rho ^{-2K} +(1- \rho)(2 \rho )^{-K} - \rho(2- \rho) =   \rho ^{-2K}  \Big (1+(1- \rho) (\rho /2)^K \Big ) -  \rho(2- \rho)  \\
  =  \rho ^{-2K}  \Big (1 - (\rho /2)^K  +(2- \rho) (\rho /2)^K \Big ) -  \rho(2- \rho)  \\ 
  =    (2- \rho)   \rho ^{-2K} \left ((\rho /2)^K - \rho ^{2K+1}   + \frac{1}{2} \sum _{k=0}^{K-1} (\rho/2) ^k  \right )  \\=   (2- \rho)   \rho ^{-2K} \left (  ( \rho /2 -1) \sum _{k=0}^{K-1} (\rho/2) ^k    - (\rho -1)  \sum _{k=0}^{2K} \rho ^k     +  \frac{1}{2} \sum _{k=0}^{K-1} (\rho/2) ^k \right )  \\
  =   (1- \rho)  (2- \rho)   \rho ^{-2K} \left (  \sum _{k=0}^{2K} \rho ^k  - \frac{1} {2}\sum _{k=0}^{K-1} (\rho/2)^k  \right ).   
   \end{align*}

  \end{proof} 
   
 \begin{remark} \label{empty}
   Note that  the proof of Theorem \ref{blocking} additionally provides, through equation \eqref{A_K(1)}, the stationary probability $A_K(1)$ that queue $1$ (resp.~$2$)  is empty. 
   \end{remark}

 \begin{remark} \label{degree_two}
  The  property, here used for $y= 1/ \rho$,  that equation $p_x(y)=0$ has  in general two solutions   $x\in \C$ for given $y$ --because $p_x(y)$ is a degree-two polynomial with respect to $x$-- will be used in the next section for building infinite chains of  coupled roots. 
   \end{remark}
   
   We can extend our result to a system where there is a constraint on the total number  of customers in the system, rather than separate constraints for each queue. To solve both cases of even and odd total capacity, we need to consider the following variant of our original model. Here again the queues have capacity $K$, but the system cannot accept more than $2K-1$ customers.  The state space is thus reduced to $\{0, \dots , K \} ^2 \setminus \{ (K,K) \}$ and the transitions and rates are the same as previously, except that those  from  $(K-1,K)$ and $(K,K-1)$ to $(K,K)$ and vice versa no longer exist. Denoting  $\widetilde \pi _K$ the stationary distribution, the stationary blocking probability is given by $\widetilde \pi _K (K-1,K) +  \widetilde \pi _K (K,K-1)  = 2 \, \widetilde \pi _K (K-1,K) $ and can be determined by following  the same steps as for $\pi_K(K,K)$. The result is given  in the next theorem.

  \begin{theorem} \label{tilde_blocking} For $\rho >0$, 
\begin{multline*} \label{tilde_pi(K,K)}  2\, \widetilde \pi _K (K-1,K)  =  \frac {2 \rho ^{2K-1}} {2\sum _{k=0}^{2K-1} \rho ^k - \sum _{k=0}^{K-1} (\rho/2) ^k}  \\
=  \frac {(1- \rho)(2 - \rho )}{  \rho ^{-(2K-1)} + \rho (1- \rho)(2 \rho )^{-K} - \rho (2 - \rho)} \quad \text{if }\ \rho \ne 1,2.   
\end{multline*} 
\end{theorem}

\begin{proof}
It is similar to that of Theorem \ref{blocking}.  Defining for $x,y \in \C$, 
\[ \widetilde F_K(x,y)~\stackrel{def}{=}~ 
    \sum _{ 0 \le j \le k \le K,  \, (j,k) \neq (K,K) } \widetilde \pi_K(j,k) \, x^j y ^{k-j}  \] 
    \[ \widetilde A_K(y) =  \sum _{k=0}^K \widetilde \pi_K(0,k) \, y^k \quad \text{ and } \quad \widetilde B_K(x) =  \sum _{j=0}^{K-1}\widetilde  \pi_K(j,j) \, x^j, \]
 we prove the following relation, analogous to \eqref{F_K}: For  $x,y \in \C$,
 \begin{multline*}   \Big ( y^2 - 2(1+ \rho) xy + (1+2 \rho x) x  \Big) \, \widetilde F_K(x,y)  ~=~ y(y-x) \, \widetilde A_K(y)  \\
   ~-~ \Big (  \rho y^2+ (1+ \rho) y - 1-2 \rho x  \Big)  x \widetilde B_K(x)  ~+~ 2 \rho \,  x^{K} y(x-y) \, \widetilde \pi_K(K-1,K), 
   \end{multline*}
    from which we derive that, if  $y,z$ are the  roots  of any polynomial $p_x$, then
    \[
 \phi (y,z) \widetilde A_K(y) - \phi(z,y) \widetilde A_K(z) ~=~ 2 \rho \,  (1+ \rho) \,  x^{K}  (y-z) \, \widetilde \pi_K(K-1,K).
 \]
 The theorem then follows by using the same particular values for $x$ as in the proof of Theorem \ref{blocking}.

\end{proof}

\medskip
From Theorems \ref{blocking} and \ref{tilde_blocking}, one can derive the stationary blocking probability of  a system of two identical $M/M/1$ queues under JSQ, but  with the constraint that the total number of customers in the system cannot exceed some value $M $. Indeed, since under JSQ the maximum of the two queues can increase only when both are equal,  the associate queue-length Markov process is \emph {not irreducible}: Defining
\[ S'_M \stackrel{def}{=} \left \{ \begin{array}{lll}  \{0, \dots , M/2 \}^2 & \text { if } M \text { is even,} \\
 & \\
\{0, \dots , \frac{M+1}{2} \}^2 \setminus \{ \left (\frac{M+1}{2}, \frac{M+1}{2} \right) \} & \text { if } M \text { is odd,}
\end{array} \right .
\]
then for all $M$, the set $S'_{M}$ is closed under the dynamics and the process eventually ends its life in this set.  Once in $S'_{M}$, the process behaves as the standard JSQ for even $M$, and as the variant in Theorem  \ref{tilde_blocking} for odd $M$. Since  in both cases, $S'_{M}$ is the only absorbing irreducible component, the process has a unique stationary distribution given by that of JSQ (or its variant) in $S '_{M}$. The stationary blocking probability is then given by $\pi _{M/2}(M/2,M/2)$ if $M$ is even, and by  $2 \,  \widetilde \pi _{\frac{M+1}{2} }\left (\frac{M-1}{2} , \frac{M+1}{2} \right )$ if $M$ is odd. 

Note  that the same result holds for a system of two identical queues with a double constraint $\max(L_1,L_2) \le K$ and $L_1 + L_2 \le M$, where $M < 2K$. The system eventually ends in the set $S'_M$,
and has the same steady state as under
 the sole constraint  $L_1 + L_2 \le M$.

 \subsection{Asymptotics of the blocking probability and comparison with related models.} \label{subsec2.2}
  
  Asymptotics of the blocking probabilities are given by Proposition \ref{asymptotics}. The JSQ system is then compared with both systems of two independent  $M/M/1/K$ queues and  one two-server $M/M/2/2K$ queue.  Here all servers have rate $1$, each one-server queue has arrival rate $\rho$, while the two-server queue has arrival rate  $2 \rho$. Such a comparison is natural. Indeed, the first system of two independent queues  corresponds to two queues without cooperation.    Its blocking probability is the common blocking probability of the two queues. (One may indeed consider that customers arrive according to a global Poisson flow with parameter $2 \rho$ and  choose either of the two queues with equal probability.)  On the contrary, the second system   is the case of fully shared resources: The two servers serve the total flow of customers. 
   A stochastic ordering is established in Proposition \ref{couplings}. It yields inequalities between the different blocking probabilities, whose differences are, moreover, controlled by Proposition \ref{uniform_bounds}.

 \begin{proposition} \label{asymptotics}
The following  asymptotics hold.
For fixed $K$,
  \begin{align*}
   \pi_K(K,K)=\begin{cases}
    2 \rho ^{2K} \! +  O\! \left (\rho ^{2K+1} \right ) & \text{ as } \rho \to 0,\\
1 - 1/ \rho   +  O  \left (\rho ^{-(K+1)} \right )& \text{ as } \rho \to + \infty .
\end{cases}
\end{align*}
For fixed $\rho$, as $K$ tends to $+\infty$,
 \begin{align*}
   \pi_K(K,K)=\begin{cases}
      \rho ^{2K} (1- \rho)(2 - \rho)+ O \big ((\rho ^3/2)^K \big ) &  \text{ if } \rho < 1/2,\\
      \rho ^{2K} (1- \rho)(2 - \rho)+ O \big (\rho ^{4K} \big ) &  \text{ if } 1/2 \le \rho <1,\\
      (2K) ^{-1} + O \big ( K^{-2}\,  2 ^{-K}  \big ) &  \text{ if } \rho =1,\\
     1- 1/\rho+  O \big (\rho ^{-2K}\big )  & \text{ if } 1 <\rho < 2, \\
      1/2  + O \big (K \,  2 ^{-2K}  \big ) &  \text{ if } \rho =2,\\
1- 1/\rho +  O  \big  ((2\rho) ^{-K}  \big  ) & \text{ if } \rho > 2.   \\
 \end{cases}
\end{align*}   
\end{proposition}
\begin{proof}
It is easily obtained from the explicit expression for $\pi(K,K)$ given in Theorem~\ref{blocking}.
\end{proof}

 The JSQ model can be coupled  with  the $M/M/2/2K$ queue in such a way that the total number of  customers always remains larger in  JSQ  than in the other system. 
Actually, the difference  between the two appears only  when one of the queues becomes empty in the JSQ model. In this case,  one server idles in the JSQ model, while,  if some customer is waiting in the other queue, the server immediately begins a new service  in the  coupled $M/M/2/2K$ queue. The intuition is that this difference is negligible if the traffic is not very low, and if it is, the blocking probabilities are both very small.  The conclusion of the section is that  the blocking probabilities for the JSQ policy and  for the  $M/M/2/2K$ queue are indeed very close for any range of values of $\rho$. 

On the other hand,  JSQ can be coupled with one $M/M/1/K$ queue with arrival rate $2 \rho$ and service rate $2$, in  such a way that the latter queue dominates the amount of customers exceeding  $K$ in JSQ. Here, the difference lies in the fact that  the total number of customers in JSQ can become smaller than $K$.
 Nevertheless, the blocking probabilities, as functions of $\rho$, turn out to be uniformly close as $K$ gets large.

Let  $N(t)$ be the total number of customers present at time $t$ in the JSQ system, and let $\overline N(t)$, $\underline N(t)$ and $\underline { \underline N} (t)$ be the numbers of customers at time $t$ in, respectively, an $M/M/1/K$ queue with arrival rate $2 \rho$ and service rate $2$, an $M/M/2/2K$ queue with arrival rate $2 \rho$ and service rate $1$, and an $M/M/1/2K$ queue with parameters $2 \rho$ and $2$.


\begin{proposition} \label{couplings}
 The four processes $N, \overline N, \underline N$ and  $\underline { \underline N}$ can be coupled together in such a way that the inequalities
\begin{equation} \label{stochastic_order}  \underline { \underline N} (t) ~\le ~ \underline N(t) ~\le~ N(t) ~\le ~ K+ \overline N(t)
\end{equation}
are satisfied  at all positive times $t$ if they hold at $t=0$.


\end{proposition}

\begin{proof}
 The four processes can be built from four independent Poisson processes, $\mathcal N^i_a, \mathcal N^i_d , i=1,2$, where $\mathcal N ^1_a, \mathcal N ^2_a$  each have parameter $\rho$,  and $\mathcal N^1_d, \mathcal N^2_d$, parameter $1$. For all systems, $\mathcal N^1_a+\mathcal N ^2_a$ represents the global flow of arrivals. More particularly for the JSQ system, $\mathcal N ^i_a$ ($i=1,2$) is the part of the total flow that is directed to file $i$ in case of equality of the two files.  For the $M/M/1/K$ and $M/M/1/2K$  queues, $\mathcal N^1_d + \mathcal N ^2_d$ is the service process associated with the unique server. While for both JSQ and the $M/M/2/2K$ queue,  $\mathcal N ^1_d$ and $ \mathcal N ^2_d$ represent the service processes of each of the two servers. The following details are important for the coupling $(\underline N, N)$.  In JSQ,  $\mathcal N ^1_d$ (that is, one of the two servers) is dedicated to the queue with maximal length, or to queue $1$  in case of equality; while  $\mathcal N ^2_d$ (that is, the other server) operates at the queue with minimal length, or at queue 2  when both are equal. As for the $M/M/2/2K$ process, $\mathcal N ^1_d$ operates when there is at least one customer present, while $\mathcal N ^2_d$ only does when  at  least 2 customers are present. These choices lead to the following expressions of the increments of the different processes. First for JSQ, \begin{multline*}
d L_1(t) = \1 _{L_1(t^-) < L_2(t^-)} \big ( \mathcal N ^1_a(dt)  + \mathcal N ^2_a(dt)  \big ) +   \1 _{L_1(t^-) =  L_2(t^-) <K}\, \mathcal N ^1_a(dt) \\
 - \1 _{ L_2(t^-) \vee 1 \le  L_1(t^-)} \, \mathcal N^1_d (dt)  - \1 _{1 \le  L_1(t^-) <  L_2(t^-)} \,  \mathcal N^2_d (dt)
\end{multline*}
\begin{multline*}
d L_2(t) = \1 _{L_2(t^-) < L_1(t^-)} \big ( \mathcal N ^1_a(dt)  + \mathcal N ^2_a(dt) \big ) +   \1 _{L_1(t^-) =  L_2(t^-) <K}\, \mathcal N^2 _a(dt) \\
 - \1 _{ L_1(t^-)  < L_2(t^-)} \, \mathcal N ^1_d (dt)  - \1 _{1 \le  L_2(t^-) \le  L_1(t^-)} \,  \mathcal N ^2_d (dt).
\end{multline*}
We get by summation
\begin{multline*}
d N(t) = \1 _{N(t^-)< 2 K}  \big  ( \mathcal N ^1_a(dt)  + \mathcal N ^2_a(dt)  \big  ) \\
 - \1 _{ N(t^-)\ge 1} \, \mathcal N ^1_d (dt)  - \1 _{
  L_1(t^-) \wedge  L_2(t^-) \ge 1} \,  \mathcal N ^2_d (dt)
\end{multline*}
Here and throughout the paper, we use the symbols $\vee$ and $\wedge$  to denote, respectively, the maximum and the minimum of two real numbers.  

Next,  for the three other processes,
\[d \underline N(t) = \1 _{\underline N(t^-)< 2 K}  \big  ( \mathcal N ^1_a(dt)  + \mathcal N ^2_a(dt)  \big  ) 
 - \1 _{ \underline N(t^-)\ge 1} \, \mathcal N ^1_d (dt)  - \1 _{\underline N(t^-) \ge 2} \,  \mathcal N ^2_d (dt), \]
  \[d \underline { \underline N} (t) = \1 _{\underline { \underline N}(t^-)<  2K}  \big  ( \mathcal N ^1_a(dt)  + \mathcal N ^2_a(dt)  \big  ) 
 - \1 _{ \underline { \underline N} (t^-)\ge 1} \, \big (\mathcal N ^1_d (dt) +  \mathcal N ^2_d (dt) \big ), \]
  \[d \overline N(t) = \1 _{\overline N(t^-)<  K}  \big  ( \mathcal N ^1_a(dt)  + \mathcal N ^2_a(dt)  \big  ) 
 - \1 _{ \overline N(t^-)\ge 1} \, \big (\mathcal N ^1_d (dt) +  \mathcal N ^2_d (dt) \big ). \]
 From these expressions, it is easily proved that, for any increasing point $t$ of either of the processes $\mathcal N ^i_a, \mathcal N ^i_d , i=1,2$, if  the  inequalities $\underline { \underline N}(t^-) \le \underline N(t ^-) \le N(t^-) 
 \le  K+ \overline N(t^-)$ hold, then they are still valid at $t$. That is, the inequalities  \eqref{stochastic_order}  are preserved in time.  Indeed,  $N, \underline N, \underline { \underline N}$ and $\overline N$  have jumps $\pm 1$, and no upward jump of one process can coincide with a downward jump of another one. So, for example, the first inequality $\underline { \underline N}(t^-) \le \underline N(t ^-) $ can turn into $ \underline { \underline N}(t) > \underline N(t ) $ only if $\underline { \underline N}(t^-) = \underline N(t ^-)$ and $d \underline { \underline N}(t)  > d \underline N(t )  $. Same for the other inequalities. But this clearly cannot occur at increasing points of  $ \mathcal N ^1_a$ or $ \mathcal N ^2_a$ (if some  equality holds at $t^-$,  both the considered processes undergo the same positive jumps from $ \mathcal N ^1_a$ or $ \mathcal N ^2_a$). As for increasing points of  $\mathcal N ^1_d$,  they preserve both $\underline { \underline N} \le \underline N\le N$,  for analogous reasons, and  $N \le  \overline N +K$ because $N (t^-) = \overline N(t^-) +K$ implies that $N(t^-) \ge 1$.  Lastly, at increasing points of $\mathcal N ^2_d$,  the inequality  $\underline { \underline N} \le \underline N$ is clearly preserved. The same holds for inequality $ \underline N\le N$, because if   $\underline N(t ^-) = N(t^-) =  L_1(t^-) +  L_2(t^-)$ and  $ L_1(t^-) \wedge  L_2(t^-) \ge 1$, then $\underline N(t ^-)  \ge 2$. And  $N \le  \overline N +K$  is   also preserved because if  $ N (t^-) = \overline N(t^-) +K$ and $ \overline N(t^-)\ge 1$, then $L_1(t^-) +  L_2(t^-) \ge K+1$, and so, necessarily, $ L_1(t^-) \wedge  L_2(t^-) \ge 1$.  The proof of  \eqref{stochastic_order} is complete.

\end{proof}


\medskip
For $\rho >0$ and $K \ge 1$, we denote by $\nu _K$ the geometric distribution with parameter $\rho$ on $\{0,\cdots,K\}$, and by $\nu'$ the stationary distribution  of the $M/M/2/2K$ queue with arrival rate $2\rho$ and service rates $1$, given by
\begin{equation} \label{nu'}
\qquad \nu'(0)= \frac{1}{2\sum _{k=0}^{2K} \rho ^k - 1} 
\quad \text{ and } \quad  \nu'(n) =
2\nu'(0)\rho^n \  \text{ for } 1\leq n\leq 2K.
\end{equation}

 It results  from Proposition \ref{couplings} that  
\begin{equation}  \label{order} \nu_{2K} (2K) \le \nu '(2K) \le \pi_K(K,K) \le \nu_K (K) .
\end{equation}

Note that those inequalities are readily recovered from the following explicit values of the different blocking probabilities: For $\rho >0$,
\[ \nu _K (K) = \frac{\rho ^K} {\sum _{k=0}^K \rho ^k}, \qquad   \nu '(2K) =  \frac {2 \rho ^{2K}} {2\sum _{k=0}^{2K} \rho ^k - 1} \]
and
\[ \pi _K(K,K) =  \frac {2 \rho ^{2K}} {2\sum _{k=0}^{2K} \rho ^k - \sum _{k=0}^{K-1} (\rho/2) ^k} \cdot  \]
The next proposition provides
bounds on the uniform norms   of  $\nu_K (K) - \pi_K(K,K) $  and   $ \pi_K(K,K) - \nu '(2K)$. Those show that $\pi_K(K,K)$, as a function of $\rho$, is at distance of the order of $K  ^{-1} $ from $\nu_K (K)$  and  $K ^{-2}$ from $\nu ' (2K)$.  Note that $ \nu_K(K)$ is the blocking probability  in the system of two independent $M/M/1/K$ queues with arrival rates $\rho$ and service rates $1$.


\begin{proposition} \label{uniform_bounds}
 For  any  $K \ge 1$,
\[  \frac{K + 2^{-K} -1  }{(K+1)(2 K + 2^{-K}) } ~ \le ~  \sup _{ \rho >0} \Big ( \nu_K (K)  - \pi _K(K,K)  \Big ) ~ \le ~ \frac{1}{K+1} \,  \]
\[\frac{2^{-1} - 2^{-K} }{(2K+ 2^{-1})(2K+2^{-K} )}  \le\sup _{ \rho >0} \Big ( \pi _K(K,K)  -  \nu '(2K)  \Big ) ~ \le ~ \frac{ 2} {K^2 }\,\cdot \]

\end{proposition}
\begin{proof}
The  lower bounds are trivially obtained by taking values  at $\rho =1$.  As for the upper bounds, we first use the inequality
\[\nu_K (K) - \pi_K(K,K)  ~ \le ~ \nu_K (K) - \nu _{2K}(2K)  \qquad (\rho >0)\]
that results from \eqref{order}. This yields, for $\rho >0$,
\begin{multline*}
 \nu _K(K) - \pi_K(K,K)  ~ \le ~   \frac{\rho ^K} {\sum _{k=0}^K \rho ^k} -  \frac{\rho ^{2K}} {\sum _{k=0}^{2K} \rho ^k} 
=  \frac{\rho ^K  \sum _{k=0}^{K-1} \rho ^k }{\left (\sum _{k=0}^K \rho ^k \right ) \left (\sum _{k=0}^{2K} \rho ^k \right )}\\
\le~  \frac{\rho ^K  }{\sum _{k=0}^{2K} \rho ^k } ~\le ~   \min \left ( \frac{\rho ^K  }{\sum _{k=0}^{K} \rho ^k }\,  , \,  \frac{\rho ^K  }{\sum _{k=K}^{2K} \rho ^k } \right ) 
 ~\le ~\frac{1}{K+1} .
\end{multline*}
Here, for $\rho < 1$ we have used  $ \rho ^K / (\sum _{k=0}^{K} \rho ^k) = ( \sum _{k=0}^{K} \rho ^{-k}) ^{-1}$, and for $\rho \ge 1$, $ \rho ^K / (\sum _{k=K}^{2K} \rho ^k) = ( \sum _{k=0}^{K} \rho ^k) ^{-1} $. 

Now for  $ \pi _K(K,K)  -  \nu '(2K)$, one can compute
 \begin{multline*}
\pi _K(K,K)  -  \nu '(2K) 
=  \rho ^{2K+1}  \frac { \sum _{k=0}^{K-2} (\rho/2) ^k }{ \left (2\sum _{k=0}^{2K} \rho ^k - 1 \right )\left ( 2\sum _{k=0}^{2K} \rho ^k - \sum _{k=0}^{K-1} (\rho/2) ^k \right ) } \, \cdot 
 \end{multline*}
 Then using $\ 2 \sum _{k=0}^{2K} \rho ^k -1 ~\ge~  2\sum _{k=0}^{2K} \rho ^k -2 ~=~ 2 \rho \sum _{k=0}^{2K-1} \rho ^k $, and  \\ 
\[2\sum _{k=0}^{2K} \rho ^k - \sum _{k=0}^{K-1} (\rho/2) ^k  ~\ge~  \sum _{k=0}^{2K} \rho ^k+ \sum _{k=0}^{K-1}  \big ( \rho ^k -(\rho/2) ^k\big )  ~\ge ~ \sum _{k=0}^{2K} \rho ^k   , \]
 one gets
 \[ \pi _K(K,K)  -  \nu '(2K)  ~\le ~ \frac {  \rho ^{2K}  \sum _{k=0}^{K-2} (\rho/2) ^k }{2  \left ( \sum _{k=0}^{2K-1} \rho ^k  \right )    \left ( \sum _{k=0}^{2K} \rho ^k  \right )  } \cdot \] 
 We next consider two cases. First, if $\rho < 3/2$, then $\sum _{k=0}^{K-2} (\rho/2) ^k \le \sum _{k=0}^{\infty } (3/4) ^k = 4,$ so that
  \[ \pi _K(K,K)  -  \nu '(2K)  ~\le ~ 2\left (\frac {  \rho ^{K}  }{ \sum _{k=0}^{2K-1} \rho ^k } \right )  ^2 \cdot \] 
 Analogously to the argument previously used for $ \nu (K) - \pi_K(K,K) $, 
\[ \frac {  \rho ^{K}  }{ \sum _{k=0}^{2K-1} \rho ^k } ~\le ~ \min  \left ( \frac{\rho ^K  }{\sum _{k=1}^{K} \rho ^k }\,  , \,  \frac{\rho ^K  }{\sum _{k=K}^{2K-1} \rho ^k } \right )  ~\le ~\frac{1}{K} ,\]
 where the last step results from  considering both cases $\rho <1$ and $1 \le \rho <3/2$. One gets, for  $\rho < 3/2$,
  \[ \pi _K(K,K)  -  \nu '(2K)  ~\le ~ \frac{2}{K^2} .\]
  Now for $\rho \ge 3/2$, using $ \sum _{k=0}^{K-2} (\rho/2) ^k  ~\le ~ \sum _{k=0}^{K-1} \rho ^k$ and
 $\displaystyle{  \sum _{k=0}^{2K-1} \rho ^k  ~=~ (1 + \rho ^K) \sum _{k=0}^{K-1} \rho ^k}$ 
 gives
 \[\pi _K(K,K)  -  \nu '(2K)  ~\le ~  \frac {  \rho ^{2K}  }{ 2   \left (1+ \rho ^K  \right )   \left ( \sum _{k=0}^{2K} \rho ^k  \right ) } ~\le ~  \frac { 1 }{ 2   \left (1+ (3/2)^K  \right )  }  \cdot \] 
 It is easily checked that $\displaystyle{\frac { 1 }{ 2   \left (1+ (3/2)^K  \right )  } 
 \le \frac{2}{K^2}}$ holds for all $K \ge 1$,  and this completes the proof. 

\end{proof}

Figure~\ref{blocking_proba_plot} shows that, as regards the relative error on the blocking probability, only the approximation by the $M/M/2/2K$ queue  is satisfactory for all values of $\rho$.  Nevertheless, for $\rho>1$, as $K$ grows, both approximations become accurate.

\begin{figure} 
\centering
\begin{tikzpicture}[scale=0.7]
       \begin{axis}[
title = {$K=5$},
legend style ={at={(0.98,0.20)}, anchor=north east, draw=none, fill=none},
xlabel={$\rho$},
ylabel={Ratio of blocking probabilities},
font=\large,
label style={font=\large},
xmin=0, ymin=0.85, xmax=2, ymax=1.15
]
\addplot[blue, thick, smooth, samples = 80, domain = 1.01:1.99]
	{-((1/32)*x-x^7-1/32-1/x^5+2*x^6)/((x^6-1)*(x-2))};
\addplot[red, thick, smooth, samples = 80, domain = 0:1.99] 
	{(-(1/16)*x^6+2*x^12+(1/16)*x^5-4*x^11+2)/((2*x^11-x-1)*(x-2))};
 	\addplot[black, dashed, domain=0:2, samples=2]{1};
 \end{axis}
\end{tikzpicture}
\begin{tikzpicture}[scale=0.7]
       \begin{axis}[
title = {$K=30$},
legend style ={at={(0.98,0.10)}, row sep=8pt, anchor=south east, draw=none, fill=none},
xlabel={$\rho$},
font=\large,
label style={font=\large},
xmin=0, ymin=0.85, xmax=2, ymax=1.15
]
\addplot[blue, thick, smooth, samples = 50, domain = 1.01:1.99]
	{(-(1/1073741824)*x+x^32+1/1073741824+1/x^30-2*x^31)/((x^31-1)*(x-2))};
     \addlegendentry{$\dfrac{\nu_K(K)}{\pi_K(K,K)}$ };
\addplot[red, thick, smooth, samples = 80, domain = 0:1.99] 
	{-((1/536870912)*x^31-2*x^62-(1/536870912)*x^30+4*x^61-2)/((2*x^61-x-1)*(x-2))};
     \addlegendentry{$\dfrac{\nu'(2K)}{\pi_K(K,K)}$ };
     \addplot[black, dashed, domain=0:2, samples=2]{1};
 \end{axis}
      \end{tikzpicture}      
      
\caption{Ratios of blocking probabilities of the $M/M/1/K$ and $M/M/2/2K$ queues with respect to the JSQ blocking probability as a function of the arrival-to-service rate ratio $\rho$, for $K=5$ and $K=30$.}
\label{blocking_proba_plot}
\end{figure}
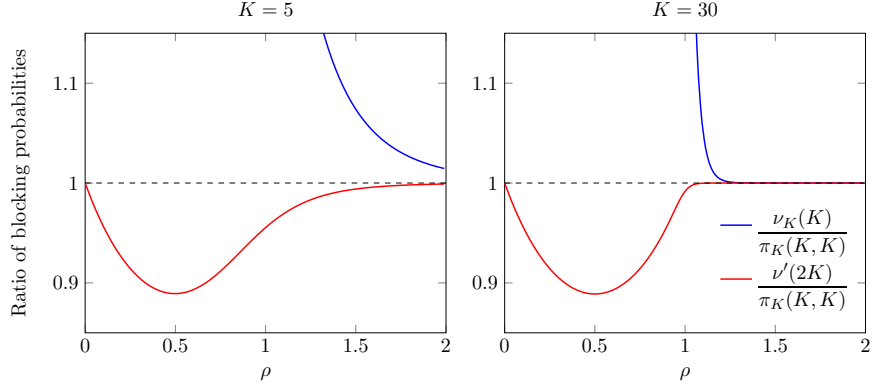

\subsection{The total number of customers in the system}

An upper bound on the stationary mean total number of customers in the JSQ system is derived from the  following lemma. Note that  the  probabilities at  values larger than or equal to $K$ are here explicit.

\begin{lemma}\label{lemma_bound}
The total number of customers $N_K$ at steady state in the JSQ model is such that
\begin{align*}
\mathbb P (N_K = n) &\leq \dfrac{\pi_K(K,K)}{\rho^{2K}} \rho^n, \quad \;0 \leq n < K,\\
\mathbb P (N_K = n) &= \dfrac{\pi_K(K,K)}{\rho^{2K}} \rho^n, \quad \;
  K \leq n \leq 2K,
\end{align*}
where $\pi_K(K,K)$ is given by Theorem \ref{blocking}.
\end{lemma}

\begin{proof}
From  equation \eqref{generating}, 
\begin{equation} \label{lemma_bound_1}
	\sum_{n=0}^{2K} \mathbb P (N_K = n) x^n = 2F_K(x^2,x)-B_K(x^2).
\end{equation}

Using equation \eqref{F_K} and the definition of $A_K$, we can show that for $x \neq 0$, $x \neq 1$, $x \neq 1/\rho$,
\begin{equation} \label{lemma_bound_2}
	2F_K(x^2,x)-B_K(x^2) = \frac{\sum_{k=0}^K \pi_K(0,k) x^k -\rho x^{2K+1}\pi_K(K,K)}{1-\rho x}.
\end{equation}

Expanding the denominator of the right-hand side of equation \eqref{lemma_bound_2} and rearranging the sums, it holds that, for $|x|<1/\rho$,
\begin{align} \label{lemma_bound_3}
\sum_{n=0}^{2K} \mathbb P (N_K=n) x^n &= 
\sum_{n \geq 0} \left( \sum_{k=0}^{n\wedge K} \pi_K(0,k)\rho^{n-k} \right)x^n - \sum_{n > 2K} \rho^{n-2K}\pi_K(K,K)\,x^n \nonumber\\
	&= \sum_{n = 0}^{2K} \left( \sum_{k=0}^{n\wedge K} \pi_K(0,k)\rho^{-k} \right)\rho^{n}x^n,
\end{align}
where we have used equation \eqref{A_K(1/rho)} to cancel the terms which have $n>2K$. Finally, by comparing both sides of equation \eqref{lemma_bound_3} it follows that, for ${0\leq n \leq 2K}$,
\begin{equation*}
\mathbb P (N_K=n) = \left( \sum_{k=0}^{n\wedge K} \pi_K(0,k)\rho^{-k} \right)\rho^{n} \leq A_K(1/\rho)\,\rho^n,
\end{equation*}
where we have equality if $n \geq K$. Using equation \eqref{A_K(1/rho)} again ends the proof.
\end{proof}
\begin{proposition}\label{proposition_Nbar}
The stationary mean number $\E(N_K)$ of customers  in the system  can be bounded as follows, for $\rho>0$, $\rho\not =1$,
\begin{align*}
 {\frac {2\rho\, \left( 1- \left( 1+2K \left( 1-\rho \right) \right) {\rho}^{2K} \right) }{ \left( 1-\rho \right)  \left( 1+\rho-2\,{\rho}^{2K+1} \right) }} \leq \E(N_K) \leq \left(2K(\rho-1)+\rho^{-2K}-1 \right) \dfrac{\rho\,\pi_K(K,K)}{(\rho-1)^2}  ,
\end{align*}
where $\pi_K(K,K)$ is given by Theorem \ref{blocking}, and for $\rho=1$,
\[\frac{K\,(2K+1)}{2K+1/2}\leq \E(N_K) \leq \frac{K\,(2K+1)}{2K+2^{-K}}.\]
\end{proposition}
\begin{proof}
The upper bound follows directly from Lemma \ref{lemma_bound} and the definition of $\E(N_K)$. The lower bound is simply the stationary mean number of customers in an $M/M/2/2K$ queue; see equation~\eqref{nu'}. It is smaller than $\E(N_K)$ by the coupling argument introduced in Section~\ref{subsec2.2}. Bounds for $\rho=1$ are obtained by extending the previous expressions by continuity. Indeed,  as mentioned in the proof of Theorem \ref{blocking},  $ ( \pi_K(j,k), 0 \le j,k \le K)$ is continuous with respect to $\rho > 0$. 
 \end{proof}
 
 In Figure \ref{nbar_lower_upper_bound}, we check the tightness of the bounds presented in Proposition \ref{proposition_Nbar}. The fact that the JSQ policy achieves a  stationary mean number of customers very close to that of the $M/M/2/2K$ queue (lower bound) is remarkable. Furthermore, when $\rho \to +\infty$, the difference between the upper and the lower bounds is of the order of $4K(2\rho)^{-K-1}+\mathcal{O}(\rho^{-K-2})$.

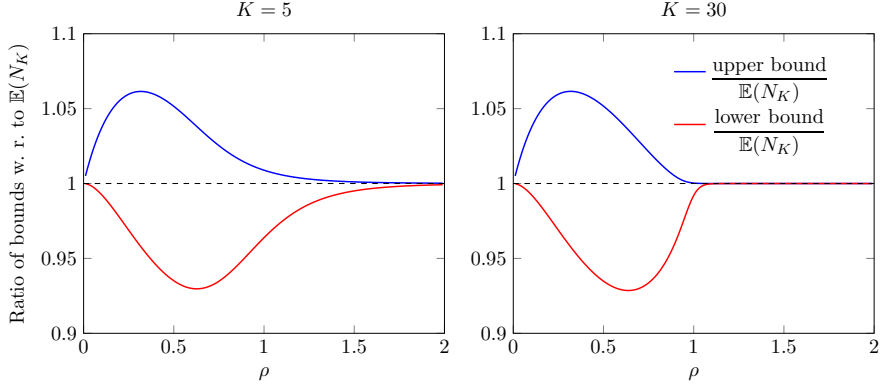
\begin{figure}[htbp] 
\centering
\begin{tikzpicture}[scale=0.7]
\begin{axis}
[
title = {$K=5$},
area legend,
legend style={at={(0.9,0.95)},
anchor=north east, draw=none, fill=none},
font=\large,
label style={font=\large},
ylabel={Ratio of bounds w.~r.~to $\E(N_K)$},
xlabel={$\rho$},
xmin=0, xmax=2,
ymin=0.9, ymax=1.1
]

\addplot[blue, line legend, thick] table
	[
		x expr = \thisrow{rho},
    	y expr = \thisrow{upper}
    ] {n_bar_bound_K_5.txt};
\addplot[red, line legend, thick] table
	[
		x expr = \thisrow{rho},
    	y expr = \thisrow{lower}
    ] {n_bar_bound_K_5.txt};
\addplot[black, dashed, domain=0:2, samples=100,unbounded coords=jump]{1};
\end{axis}
\end{tikzpicture}
\begin{tikzpicture}[scale=0.7]
\begin{axis}
[
title = {$K=30$},
area legend,
legend style={at={(0.9,0.95)},row sep=8pt,
anchor=north east, draw=none, fill=none},
font=\large,
label style={font=\large},
xlabel={$\rho$},
xmin=0, xmax=2,
ymin=0.9, ymax=1.1
]

\legend{$\dfrac{\text{upper bound}}{\E(N_K)}$,$\dfrac{\text{lower bound}}{\E(N_K)}$}
\addplot[blue, line legend, thick] table
	[
		x expr = \thisrow{rho},
    	y expr = \thisrow{upper}
    ] {n_bar_bound_K_30.txt};
\addplot[red, line legend, thick] table
	[
		x expr = \thisrow{rho},
    	y expr = \thisrow{lower}
    ] {n_bar_bound_K_30.txt};
\addplot[black, dashed, domain=0:2, samples=100,unbounded coords=jump]{1};
\end{axis}
\end{tikzpicture}
\caption{Ratios of the upper and lower bounds with respect to the stationary mean number of customers $\E(N_K)$ in the system, at equilibrium, for $K=5$ and $K=30$.}
\label{nbar_lower_upper_bound}
\end{figure}

\begin{remark}
For $K = \infty,~0 < \rho < 1$ (which is the condition for existence of a steady state),   using  Remark \ref{A(1/rho)}  of Section \ref{sec3} that comes next, one can show that
\begin{align*}
\E(N_\infty)\leq \frac{\rho\,(2-\rho)}{1-\rho}.
\end{align*}
This is, in fact, a tighter upper bound and has a simpler form than the one presented in \cite{Halfin-1}, which is valid only for $\frac{1}{2}\le\rho<1$ (see Figure~\ref{compare_halfin}).
\end{remark}

To prove this inequality we need an equivalent form of Lemma~\ref{lemma_bound} for the infinite capacity case, which is
\begin{equation} \label{PN_infty}
	\mathbb P(N_\infty=n) \leq (2- \rho)(1- \rho)\,\rho ^n,
\end{equation}
for $0 < \rho < 1,~n \geq 0$. Then, the result about $\E(N_\infty)$ is immediate.
To prove equation~\eqref{PN_infty}, we start from
\begin{equation*}
\mathbb P (N_\infty=n) = \left( \sum_{k = 0}^n \pi_\infty(0,k)\rho^{-k} \right)\rho^{n} ,
\end{equation*}
which can be found by following the same steps as in Lemma~\ref{lemma_bound}, but for  infinite capacity. 
Now, Remark \ref{A(1/rho)} in Section \ref{sec3} will show that  
\begin{equation*}
\sum_{k = 0}^{\infty} \pi_\infty(0,k)\rho^{-k}= (2-\rho)(1-\rho),	
\end{equation*}
from which equation~\eqref{PN_infty} immediately follows.
\qed

Regarding a lower bound for $\E(N_{\infty})$, it is easily shown that the coupling argument between JSQ and the $M/M/2$ queue extends to  infinite capacity. This yields the lower bound
 presented in \cite{Halfin-1}, which is given by $2\rho/(1-\rho^2)$.
Figure~\ref{compare_halfin} shows the ratios between the different bounds and $\E(N_\infty)$.

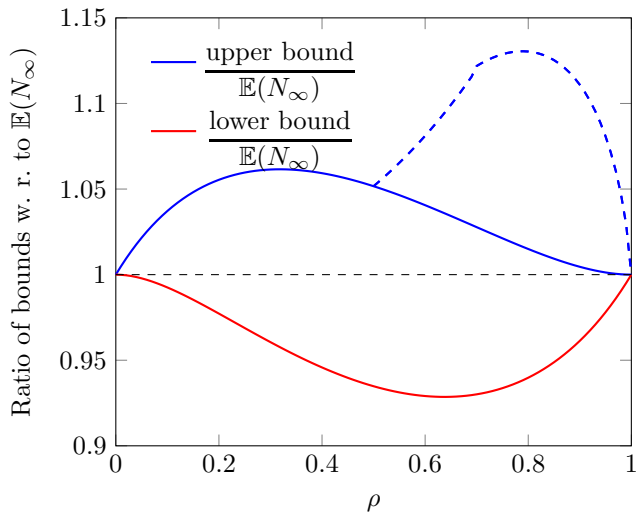
\begin{figure}[htbp]
\centering
\begin{tikzpicture}[scale=1.0]
\begin{axis}
[
area legend,
legend style={at={(0.05,0.97)},row sep=7pt,
anchor=north west, draw=none, fill=none},
ylabel={Ratio of bounds w.~r.~to $\E(N_{\infty})$},
xlabel={$\rho$},
xmin=0, xmax=1,
ymin=0.9, ymax=1.15
]
\legend{$\dfrac{\text{upper bound}}{\E(N_\infty)}$,
	$\dfrac{\text{lower bound}}{\E(N_\infty)}$}
\addplot[blue, line legend, thick] table
	[
		x expr = \thisrow{rho},
    	y expr = \thisrow{upper}
    ] {compare_halfin.txt};
\addplot[red, line legend, thick] table
	[
		x expr = \thisrow{rho},
    	y expr = \thisrow{lower}
    ] {compare_halfin.txt};
    \addplot[blue, dashed, line legend, restrict x to domain=0.5:1, thick] table
	[
		x expr = \thisrow{rho},
    	y expr = \thisrow{upperHalfin}
    ] {compare_halfin.txt};
\addplot[blue, dashed, line legend, restrict x to domain=0.5:1, thick] table
	[
		x expr = \thisrow{rho},
    	y expr = \thisrow{upperHalfin}
    ] {compare_halfin.txt};
\addplot[black, dashed, domain=0:1, samples=2,unbounded coords=jump]{1};
\end{axis}
\end{tikzpicture}
\caption{Ratios between the bounds and $\E(N_\infty)$. The dashed curve corresponds to the bound derived in \cite{Halfin-1}.}
\label{compare_halfin}
\end{figure}

\medskip
Note that    Adan et al.~\cite{adan1994upper} provide lower and upper bounds,  in the infinite capacity case, for the mean waiting time and mean number of customers.  These bounds result from  stochastic ordering of JSQ with two related systems: the shortest queue models with, respectively, threshold jockeying and threshold blocking.
The authors claim that those systems are easier to  analyze, using the matrix geometric method developed by Neuts, and provide numerical estimations.

     \subsection{The other stationary  probabilities.} \label{other}
   The Markov process $(L_1(t), L_2(t))_{t \ge 0}$ has the following particularity. Inside the upper triangle $\{ (j,k) \in \mathcal S _K,~ j \le k \}$, it has no upward jumps --or jumps to the north-- except from sites $(j,j)$ with $j < K$.  This makes it possible to solve the subsystem of balance equations at sites  $(j,k)$ with $j \le k-2$, in such a way as to express the stationary probabilities  $\pi _K(j,k)$ for $j \le k-1$  as a function only of the unknown  $\pi_K(0,k)$, $k=0, \dots , K$. Using the remaining balance equations, the stationary probabilities $\pi_K(j,j)$ are  derived in a similar form and the $\pi_K(0,k)$, $k= 0, \dots , K,$ are finally characterized recursively.
   
   In the sequel, the symbol $\ast$ will denote the discrete convolution product, defined for  any  complex-valued functions $\varphi$ and $\psi$ on $\N$ by
   \[ (\varphi \ast  \psi) (n) = \sum _{m=0} ^n \varphi(m)  \psi (n-m), \qquad n \in \N, \]
   and  for $k \ge 1$, $\varphi  ^{ \ast k} $  represents the $k$-fold convolution of $\varphi$ with itself. 
   The following relation will be used
 \begin{equation} \label{translation} \tau (\varphi \ast \psi) = (\tau \varphi ) \ast \psi + \varphi (0) \tau \psi, 
 \end{equation}
  for any $\varphi$ and $\psi$ on $\N$,
  where $\tau$ denotes the translation operator on $\C^{\N}$, defined by
\[ ( \tau \varphi) (n) = \varphi (n+1), \quad n \in \N. 
\]

   \begin{theorem} \label{pi_K} (i) For $0 \le j < k \le K$,
 \begin{equation}\label{pi_hors_diago} \pi _K( j,k)=  \sum _{l=k}^K  \pi_K(  0,l)  \Big ( g ^{\ast (l-k+1)} (j) - g ^{\ast (l-k+1)} (j+1)  \Big ) 
 \end{equation}
   and
     \[ \pi_K( k,k) = - \frac{1}{ \rho }   \sum _{l=k+1}^{K} \pi_K( 0,l)  \,  g ^{ \ast {(l-k)}}  (k+1) 
 \quad  0\le k < K, \] 
 
\noindent  where
\[ g(i) = -  \frac { \xi _+ ^i -  \xi _- ^i}{ \xi _+ - \xi _-}  \quad (i \in \N) \qquad \text {and } \quad    \xi _{\pm} = 1+ \rho \pm \sqrt {1+ \rho ^2} . \]

\smallskip
(ii) The stationary probabilities $\pi_K(0,l)$ for $l = 0, \dots ,K$ are characterized by the  following relations
\begin{equation} \label {pi(0,K)} \pi_K(0,K) ~=~\frac{\pi_K(K,K)}{2 \rho \left ( g(K-1)-g(K) \right ) }
\end{equation}
 and for $k=0, \dots , K-1$,
\begin{equation} \label  {recursion} \sum _{l=k}^{K} \pi_K(0,l) \, \left (g ^{\ast (l-k+1)} (k+2) -(2+ \rho) \, g ^{\ast (l-k+1)} (k+1)   \right ) =0 
\end{equation}

   \end{theorem}
   
   \begin{proof} The balance equation for each site $(j,k)$ with $j\le k-2$ is 
\begin{multline*}  \pi_K(j+1,k)       ~-~ \left (\1 _{\{j>0\}} + 1 + 2\rho \right )    \pi_K(j,k)   ~+~  2\rho \, \1 _{\{j>0\}} \,   \pi_K(j-1,k)\\
~= - \1 _{\{k<K\}} \,  \pi_K(j,k+1).
\end{multline*}

For  $k=K$, the right-hand side is zero and equations  at $1 \le j \le K-2$  amount to a homogeneous two-step linear recursion relation, which is solved in terms of the roots $\xi_{\pm}$ of the polynomial
$X^2 -2(1+ \rho) X + 2 \rho $. Using the boundary condition given by the balance equation at $(0,K)$, this finally yields,  for $0 \le j \le K-1$, 
\begin{multline} \label {top-line} \pi_K(j,K) ~=~  \pi _K(0,K) \, \frac {(\xi _+ -1) \xi _+ ^j + (1- \xi _-)  \xi _- ^j}{ \xi _+ - \xi _-} \\
=~ \pi _K(0,K)  \big ( g(j)-g(j+1) \big ) .  
\end{multline}
The value of $\pi_K(0,K)$ is actually known from $\pi_K(K,K)$, given by Theorem \ref{blocking}, together with the balance  equation at $(K,K)$ : $\pi_K(K,K) = 2 \rho\,  \pi_K(K-1,K)$ and formula  \eqref{top-line} at $j = K-1$. This leads to   \eqref {pi(0,K)}.

We now solve the balance equations at $(k,j)$, for some fixed $k <K$ and $0 \le j \le k-2$.
First,  from the definition of $g$ as a  linear combination of the two sequences $(\xi _+ ^j)$ and $(\xi _- ^ j)$ we have 
that 
\begin{equation} \label{fundamental} \tau ^2 g - 2(1+ \rho )  \, \tau g + 2 \rho \,  g ~=~0. 
\end{equation}
Convolution  by  arbitrary $\psi \in \C ^{ \N}$ yields
\begin{equation} \label{fundamental_bis}
\tau ^2 (g \ast \psi) - 2(1+ \rho )  \, \tau (g \ast \psi ) + 2 \rho \,  g \ast \psi ~=~  - \, \tau  \psi .\
\end{equation}
Here we have used relation \eqref{translation}, together with $g(0) = 0$ and $g(1) = -1$, which imply
 \[(\tau g) \ast \psi =  \tau (g \ast \psi)  \quad \text{   and } \quad 
(\tau ^2 g) \ast \psi =  \tau ^2 (g \ast \psi) + \tau \psi . \] 
In particular, for $\psi$ defined by 
\[ \psi (j) ~=~   \pi _K(j, k+1) \quad  \text {for } 0 \le j \le k-1,     \quad \text{arbitrary   for }  j  \ge k  , \]
we get  a two-step   linear recursion   for $  u \stackrel{def} {=} g \ast \psi$, or, abusing notation,
$ u = g \ast \pi_K( \cdot, k+1)$:
\[u (j+1) - 2(1+ \rho ) u(j) + 2 \rho u(j-1) ~=~ - \, \pi_K(j,k+1) \quad \text{ for } j = 1, \dots , k-2.\]
From the balance equations the same  relations hold with $\pi _K(\cdot \, , k )$ in place of $u$, so both functions restricted to $ \{ 0, \dots , k-1 \} $  must  differ   only by  some linear combination of the two sequences $(\xi _+ ^j)$ and $(\xi _- ^ j)$. Using the balance equation at $(0,k)$, one  finally gets, for $j = 0, \dots k-1$,  
\[\pi_K( j,k) ~=~ \pi_K( 0,k)    \Big (g(j) - g(j+1) \Big ) ~+~      \Big (   \pi_K(\cdot   , k+1) \ast g  \Big ) (j) ,\] 
or in other words,
\[\pi_K( \cdot \,  ,k) ~=~ \pi_K( 0,k)     ( g -  \tau g  ) ~+~       \pi_K(\cdot   , k+1) \ast g   .\] 
Iteration, together with  relation \eqref{translation}, yields for $k \in \{0, \dots , k-1\}$
\[   \pi( \cdot \, , k ) =    \pi(\cdot \, ,K) \ast g^{\ast (K-k)} + \sum _{l=k}^{K-1} \pi(0,l)\cdot  (g - \tau g) \ast g^{\ast (l-k)}.
\]
 Here  $g^{\ast (0)} (0) =1$ and $g^{\ast (0)}(j)=0 $ for $j \ge 1$. Finally using  \eqref{top-line}  proves   \eqref{pi_hors_diago}. 

 To derive the diagonal values $\pi _K(k,k)$, it is convenient to use the following relation
 \begin{equation} \label{diago} \pi _K (k,k) ~=~  \frac{1}{  \rho   }\, \sum _{j=0}^{k} \pi _K (j,k+1), \qquad  k = 0, \dots , K-1, 
 \end{equation}
 that results from summing up all balance equations in the square $\{ 0, \dots , k \} ^2$.
 Indeed, for any  subset $D$ of the state space $\mathcal S _K$, summation of the balance equations at sites in $D$ yields
  \[ \sum _{(l,l') \in D \times (\mathcal S _K \setminus D) } \pi _K (l) \, Q_K(l,l')
~=\sum _{(l,l') \in   (\mathcal S _K \setminus D) \times D } \pi _K  (l) \, Q_K(l,l'). \]
The  simple form of the relation for  $D = \{ 0, \dots , k \} ^2$ is due to the fact that there are only two jumps  to the outside of $D$, namely, from $(k,k)$ to $(k, k+1)$ and to $(k+1, k)$.
Relations \eqref{diago} and \eqref{pi_hors_diago} yield
\begin{multline*}  \pi _K (k,k) ~=~  \frac{  1}{\rho } \sum _{l=k+1}^K   \pi_K(  0,l)  \sum _{j=0}^{k}  \Big ( g ^{\ast (l-k)} (j) - g ^{\ast (l-k)} (j+1)  \Big )\\
=~   \frac{  1}{\rho } \, \sum _{l=k+1}^K   \pi_K(  0,l) \Big ( g ^{\ast (l-k)} (0) - g ^{\ast (l-k)} (k+1)  \Big ) ~=~ - \frac{1}{ \rho }   \sum _{l=k+1}^{K} \pi_K( 0,l)  \,  g ^{ \ast {(l-k)}}  (k+1) 
\end{multline*} 
since $g(0) = 0$ implies $g ^{\ast l} (0)=0$ for all  $l \ge 1$. Part (i) of the theorem is proved.

As for (ii), relation \eqref{pi(0,K)} has already been proved.  Next, using relation \eqref{diago} together with  the balance equation at $(k,k)$, one gets for $0 < k <K$,
\begin{equation} \label{subdiago} \pi _K (k-1,k) ~=~  \frac{1}{ 2 \rho ^2   }\, \left ( \pi_K(k, k+1) + (1+ \rho)  \sum _{j=0}^{k-1} \pi _K (j,k+1) \right )  .
 \end{equation}
 Now, using \eqref {pi_hors_diago} on both sides, we get (recall that $g ^{\ast l} (0)=0$ for  $l \ge 1$)
 \begin{multline*}
2 \rho ^2  \sum _{l=k}^K  \pi_K(  0,l)  \Big ( g ^{\ast (l-k+1)} (k-1) - g ^{\ast (l-k+1)} (k)  \Big ) \\
 ~=~ -  \sum _{l=k+1}^K  \pi_K(  0,l)  \Big (  \rho \,  g ^{\ast (l-k)} (k)   + g ^{\ast (l-k)} (k+1) \Big ), 
\end{multline*} 
 or, equivalently,
 \begin{multline*}
 \sum _{l=k}^K  \pi_K(  0,l) \Big ( 2 \rho ^2  \Big ( g ^{\ast (l-k+1)} (k-1) - g ^{\ast (l-k+1)} (k)  \Big )   \\
+ \1_  { \{ l>k\}} \Big ( \rho \,  g ^{\ast (l-k)} (k)  +  g ^{\ast (l-k)} (k+1) \Big )  \Big ) ~=~0. 
\end{multline*} 
 This relation can be rewritten as \eqref{recursion}  by using relation \eqref{fundamental_bis}, here with $ g ^{\ast l}$ for $l \ge 1$ in place of arbitrary $\psi$. Indeed,   \eqref{fundamental_bis}  together with  \eqref{fundamental} gives 
\begin{equation} \label{fundamental_ter} g ^{\ast (l+1)}(k+2)  - 2(1+ \rho)\,  g ^{\ast (l+1)} (k+1) + 2 \rho\,  g ^{\ast (l+1)}(k) ~=~- \1 _{ \{ l>0\}} \, g ^{\ast l}(k+1), 
\end{equation}
for all $k \in \N$.   It results that, for $ k >0$,
 \begin{multline*} - \1 _{\{ l>0\} } \left ( \rho   \, g ^{ \ast {l}}  (k) +  g ^{ \ast {l}}  (k+1)  \right ) \\
 =  \rho \,  g ^{\ast (l+1)} (k+1) -2 \rho   \,  (1+ \rho) \, g ^{\ast (l+1)} (k) +2 \rho ^2 \, g ^{\ast (l+1)} (k-1) \\
 + g ^{\ast (l+1)} (k+2) -2(1+ \rho) \, g ^{\ast (l+1)} (k+1) +2 \rho \, g ^{\ast (l+1)} (k)   \\
  =g ^{\ast (l+1)} (k+2) -(2+ \rho) \, g ^{\ast (l+1)} (k+1)  +2 \rho ^2 \left ( g ^{\ast (l+1)} (k-1)  -  g ^{\ast (l+1)} (k)  \right ) . 
  \end{multline*}
  We then derive \eqref{recursion}. The proof of the theorem is  complete. 
\end{proof}

  \begin{remark}
  The functional equation \eqref{A_K} provides an alternative to (ii) of Theorem~\ref {pi_K} by characterizing the sequence $(\pi(0,k), 0 \le k \le K)$ through its generating function $A_K$.  Indeed, since $A_K$ is a degree $K$ polynomial, it is determined by its values at $K+1$ different points. Now, as will be clear in the next section, one can build an infinite sequence $(v_n, n \ge 1)$, in   which all terms  are different, and  such that
\[ v_1 = \frac{1}{\rho}, \quad   v_2 = 1 \quad \text{ and for } \   n\ge 1, \  v_n \text{ and }  v_{n+1} \text{ are the roots of some } p_x.\]
  Moreover, $\phi(v_{n+1}, v_n) \neq 0$ for all $n \ge 1$. So, from the initial value 
  $ A_K(1/ \rho)   = \rho ^{-2K} \pi_K(K,K)$, where $\pi_K(K,K)$ is given by Theorem \ref{blocking},  iterating equation \eqref{A_K}   determines recursively all the $A_K(v_n)$, hence $A_K$.
 \end{remark}

\smallskip
To make our results as explicit as possible, we now give expressions for the convolution powers $g^{ \ast k}$.

A first expression of $g^{ \ast k}$ for $k \ge 1$ can be obtained from  $\   g = ( \xi _- - \xi _+) ^{ -1} ( h_+ - h_-)$, 
where
$h_+$ and $h_-$ denote  the  following elementary functions on  $\N$ 
\[ h_+(j)=  \xi _+ ^j \qquad \text{ and  }  \qquad h_-(j)=  \xi _- ^j, \qquad j \in \N. \]
We get, for $k \ge 1$,
 \[  g^{ \ast k} = ( \xi _- - \xi _+) ^{ -k}   \sum _{l=0}^k  (-1) ^l  {k \choose l}  h_+^{ \ast (k-l)} \ast h_-^{ \ast l}. \]
Now  for $k \ge 1$, $ h_{\pm}^{ \ast k}$ can be formulated explicitly  as 
\begin{equation} \label {h_ast} h_+^{ \ast k}  (j) = {j+k-1 \choose k-1} \, \xi _+ ^j \quad \text{ and  }  \quad  h_-^{ \ast k}  (j) = {j+k-1 \choose k-1} \, \xi _- ^j,     \  \quad  j \in \N,
 \end{equation} 
as is easily proven recursively,  using the relation
\[ \sum _{i=0}^j  {i+p \choose p} =  {i+p+1  \choose p+1 } , \qquad  j,p \in \N. \] 
  As a result, for  $k \ge 1$ and  $j \in \N$,
   \begin{multline*} g^{ \ast k}(j)  = ( \xi _- - \xi _+) ^{ -k}  \left ( {j+k-1 \choose k-1} \, \xi _+ ^j  + (-1) ^k  {j+k-1 \choose k-1} \, \xi _- ^j  \right .   \\  \left .  +  \sum _{l=1}^{k-1}  (-1) ^l  {k \choose l}   \sum _{i=0}^{j } {i+k-l-1 \choose k-l-1} \, {j-i+l-1 \choose l-1} \,  \xi _+ ^i  \xi _- ^{j-i}  \right )  . 
   \end{multline*} 
   
   \medskip
   A more concise expression can be obtained by using the alternative relation 
      \[ g(j) =  - ( h_+ \ast h_- ) (j-1) \quad \text {for } j \ge 1  \qquad \text { and  }  \quad g(0) = 0 .\]

   Indeed for $j \ge 1$,  
   \[  g(j) =  - ( \xi _+ - \xi _-) ^{-1}  ( \xi _+ ^j - \xi _-^j)  =  - \sum _{i=0}^{j-1} \xi _+^i \xi _-^{j-1-i} =  - \sum _{i=0}^{j-1} h_+(i) h_-(j-1-i). \]
  
   \medskip
  Denote by $\sigma$ the  operator on $\C ^ {\N}$ defined   by
      \[  (\sigma f) (j) = f(j-1) \quad \text { for } \quad j \ge 1 \quad  \text {  and }    \quad (\sigma f) (0) = 0,\]
      for any complex-valued function $ f $  on $ \N$.  
   Then $g$ can be written as 
   \[ g = - \sigma \, ( h _+ \ast h _-). \] 
   Note that $\sigma $ commutes with convolution :  $\sigma (\varphi  \ast  \psi ) = (\sigma \varphi)\ast  \psi$ for any $\varphi$ and $\psi$ defined on $\N$. One then gets
   \[  g^{\ast k} = (-1) ^k  \sigma ^{k} (h _+ ^{\ast k} \ast h _- ^{\ast k}), 
   \]
 that is,
 \[ g^{\ast k}(j)  =   \left \{ \begin{array}{ll} (-1) ^k  (h _+ ^{\ast k} \ast h _- ^{\ast k})(j-k) & \text{ for } j \ge k  \\  & \\    0 & \text { for } j <k.  \end{array} \right .  \] 
 Using   \eqref{h_ast}  finally gives 
  \[   \  g^{\ast k}(j)  =  (-1) ^k   \, \sum _{i=0}^{j-k }  {i+k-1 \choose k-1}  {j-i-1 \choose k-1}  \, \xi _+ ^i \, \xi _- ^{j-k-i}  \quad  \text{ for } j \ge k . 
 \]

\section{The infinite capacity model  }\label{sec3}
The model with  two infinite capacity queues is now considered.  Here, no rejection can occur. 
The  queue-length process is Markov with state space $ \N^2$.  Its $Q$-matrix coincides at each $(j,k) \in \N^2$ with any of the matrices $Q_K$ with $K > \max (j,k)$.   The process is ergodic under the condition $\rho <1$, which will be assumed to be satisfied. 
By symmetry of the dynamics, the invariant distribution $\pi$ satisfies 
\[ \pi(j,k) ~=~\pi(k,j) \quad \text{ for } j,k \in \N, \]
together with the following set  of reduced balance equations
 \begin{align} \label{balance_infty}    \begin{cases}   
   \big (\1 _{\{k>0\} } + \rho  \big ) \,    \pi (k,k) =      2 \rho \, \1 _{ \{k>0 \}} \,   \pi(k-1,k) +   \pi (k,k+1), \  \   k \ge 0,  \medskip \\  
 \left (\1 _{\{j>0\}} + 1 + 2\rho \right )    \pi (j,k) ~=~     2\rho \, \1 _{\{j>0\}} \,   \pi_K(j-1,k) ~+~  \pi (j+1,k) \\
\hspace{25mm}     ~+~    \pi (j,k+1)  ~+~ \rho \,  \1 _{\{k= j+1\}} \pi  (j,j), \quad  0  \le j <k .
\end{cases}
    \end{align}  
    
    The following extension of Theorem~\ref{pi_K}  holds.
     \begin{theorem} \label{pi} The invariant distribution  $\pi$ satisfies the following.
\[ \pi(j,k ) =   \sum _{l=k}^{2k-1} \pi(0,l) \cdot  \left (  g^{\ast (l-k+1)}(j)  -   g^{\ast (l-k+1)} (j+1) \right )   \quad \text{ for  }  \   0 \le j < k ,\]
\[\pi(k,k) =  - \frac{1}{ \rho }   \sum _{l=k+1}^{2k+1} \pi(0,l)  \,  g ^{ \ast {(l-k)}}  (k+1)  \quad \text{ for  } \   k \in \N .\]
 \end{theorem}

\begin{proof}
The balance equations at  $(j,k)$ with $0 \le j \le k-2$ are the same as for  finite $K$ with  $K >k$. Thus, the same recursion relation as in the proof of Theorem~\ref{pi_K}  :
\[\pi( j,k) = \pi( 0,k)    \Big (g(j) - g(j+1) \Big ) +      \Big (   \pi (\cdot   , k+1) \ast g  \Big ) (j)  \] 
 holds for $j = 0, \dots k-1$.
 As in the finite capacity case, this relation can be iterated from some fixed level $k$, here up to  any $K >k$. One gets 
  \[   \pi( \cdot \, , k ) =    \pi(\cdot \, ,K) \ast g^{\ast (K-k)} + \sum _{l=k}^{K-1} \pi(0,l)\cdot  (g - \tau g) \ast g^{\ast (l-k)} \quad  \text{on } \{0, \dots , k-1\}.
\]
    Now, taking the limit as $K \to \infty$ is simple, since  $g^{\ast m}(j) = 0 $ if $ j<m$, hence
   \[ (g - \tau g)  \ast g^{\ast m}(j) = 0 \quad \text{if} \quad 0 \le j<m.\]
 Actually, for  $K\ge 2k\, $ (so that $K-k > k-1$),  the expression  reduces to 
\begin{align*}
 \pi(j,k ) &=    \sum _{l=k}^{2k-1} \pi(0,l)\cdot  \left ( (g - \tau g) \ast g^{\ast (l-k)} \right ) (j)\\
& =  \sum _{l=k}^{2k-1} \pi(0,l) \cdot  \left (  g^{\ast (l-k+1)}(j)  -   g^{\ast (l-k+1)} (j+1) \right )  
\end{align*}
  for $  0 \le j < k$.
  Alternatively, one can equivalently write
  
\[  \pi(j,k ) =   \sum _{l=k}^{\infty} \pi(0,l) \cdot  \left (  g^{\ast (l-k+1)}(j)  -   g^{\ast (l-k+1)} (j+1) \right )  
\]
or 
\[\pi(j,k ) =     \sum _{l=k}^{j+k} \pi(0,l) \cdot  \left (  g^{\ast (l-k+1)}(j)  -   g^{\ast (l-k+1)} (j+1) \right ) . 
 \]
  The second part of the theorem is derived, as for finite $K$, from relation~\eqref{diago}, that  also holds  here, for $\pi$ and any $k \in \N$. 
  
 \end{proof}

There is no equivalent here to part (ii) of Theorem~\ref{pi_K}, because there is no ``top-value" $\pi(0, K)$ to start from. Actually,  if one proceeds as in the proof of \eqref{recursion},  the resulting system of equations that relate together the $\pi(0,k)$'s  no longer determines those values uniquely --it even has  an  infinite-dimensional set of solutions.

Nevertheless, the stationary probabilities $\pi(0,k)$ have been characterized in the literature, through their generating function  
    \[A(y) ~=~\sum _{k=0} ^{ \infty} \pi(0, k) \, y^k, \qquad y \in \C . \]
    The  most explicit formulation of $A(y)$ is given in \cite{cohen1995two}, through an infinite product.  It is recalled below as Theorem~\ref{cohen}. An original, simple  proof is  here moreover proposed,    that reduces the use of complex analysis tools to one uniqueness argument. 
    
     \begin{remark} \label{Kingman_asymptotics}
     The expressions for $\pi(j,k)$ for $j<k$  in Theorem \ref{pi}  have the following equivalent formulation:  Defining    $\   \,  \displaystyle{G(z) = \sum _{j=0}^{\infty } g(j) \, z^j = \frac {-z}{1- 2 (1+ \rho )z + 2 \rho z^2} }$ for $z \in \C$,
 the generating function given,  for fixed $k \ge 1$,  by
    $\   \,  \displaystyle{  \sum _{j=0}^{k-1} \pi(j,k)\,  z^j }\ \,  (z \in \C) $
     consists of the first  $k$ terms of the following generating function: 
     \[\sum _{j=0}^{\infty } z^j \sum _{l=k}^{2k-1} \pi(0,l) \cdot  \left (  g^{\ast (l-k+1)}(j)  -   g^{\ast (l-k+1)} (j+1) \right )  =   \frac {(z-1)G(z)}{z} \sum _{l=k}^{2k-1} \pi (0,l) \, G(z)  ^{l-k},\]
        or equivalently of 
    \[H_k(z) \, \stackrel{def}{=}\,  \frac {(z-1)G(z)}{z} \sum _{l=k}^{\infty} \pi (0,l) \, G(z)  ^{l-k},\]
    since  terms with index $l \ge 2k$ in $H_k(z)$ have a factor $z^k$, due to $ G(0)= g(0)  = 0$.
    From this, one can recover the asymptotics obtained by Kingman in \cite{Kingman-1}~: 
     \[\pi (j,k) \sim C (2+ \rho)^{j-k} \rho ^{2k} \quad \text  { and }\quad  \pi (k,k) \sim  C'  \rho ^{2k}, \]
    for two constants $C$ and $C'$, as $k, j$ grow to infinity with $j <k$. The second relation can be derived from the first one by using the balance equations at $(k,k)$. As for $\pi(j,k)$ for $j<k$, the asymptotic expression results from partial fraction decomposition of  the meromorphic continuation of $A$, for which the pole with smallest modulus is $(2+ \rho)/\rho ^2 $ (see \cite{Kingman-1,cohen1995two} or Theorem~\ref{cohen} below),  together with the following  computation:  
    \[\frac{G(z)}{z} \left (G(z)-   \frac{2+ \rho}{\rho ^2} \right) ^{-1} = \frac{\rho } {2(2+ \rho)}  \left (  \Big (z -    \frac{1} {2+ \rho} \Big )  \Big (z  -\frac{2+ \rho}{2\rho }  \Big ) \right )^{-1} , \]
    noting that $1/(2+ \rho )$  is the smaller pole in the last rational expression.
         
                \end{remark}

   \smallskip
    Classically,   the starting point of the analysis is a functional equation satisfied by $A$,
    that was first  derived in \cite{Kingman-1} and is stated below as Theorem~\ref{kingman}. It is analogous to equation \eqref{A_K} of the finite capacity case, for which we have  followed the same steps as  Kingman.  
Other authors (\cite{Flatto-1,cohen1995two}) rather use the relation between $A$ and $B$, where
     \[B(x) ~=~\sum _{k=0} ^{ \infty} \pi(k, k) \, x^k, \qquad x \in \C , \]
     that is the analogue for infinite capacity of our relation between $A_K$ and $B_K$  --from which  \eqref{A_K} was derived by eliminating $B_K$.
     But in our opinion, this approach makes things less readable.

    \begin{theorem} (Kingman) \label{kingman}
    
    (i) $A(y)$ is defined for all $y \in \C$ such that $|y| <1 + 2 \rho$.
    
   (ii) For  $x $ in some neighborhood of $0$ in $\C$, the roots $y,z$ of the polynomial $p_x$ defined in \eqref{p_x}
    satisfy  
    \begin{equation} \label{A}
      A(y) \, \phi(y,z)  =  A(z) \, \phi(z,y) , 
   \end{equation}
    where $\phi$ is given by \eqref{phi}. 
    \end{theorem}
    
    \begin{proof} For   $k \in \N$, define
 $\  \displaystyle{T_k = \sum _{ j=0}^{\infty }  \pi(j,j+k) =   \sum _{ h-j =k}  \pi(j, h)  }$.

 By summing all balance equations in the domain $Ê\{ (j,h) \in \N^2, \, h-j \le k\}$ for  $k \ge 1$, one gets :
\begin{equation} \label{T} (1 +2 \rho ) \,  T_{k+1}  = T_k - \pi(0,k)  \quad \text {for} \quad k \ge 1, 
\end{equation}
from which it results that 
 \[(1+2 \rho ) \,  T_{k+1}  <~ T_k \quad \text {for} \quad k \ge 1,\]
  so that
  \[  \sum _{k=0}^{ \infty} T_k \, r^k < +\infty    \quad \text {if} \quad 0 \le r < 1+2 \rho.\]
  Now (i)  follows from inequalities $\pi(0,k) < T_k$ for $k\ge 1$, that also result from  \eqref{T}.

     The proof of (ii) --see \cite{Kingman-1}--  is similar to that of equation \eqref{A_K}. 
   \end{proof}

    Note that for $k=0$ equation~\eqref{T} is replaced by
    \[ (1+2 \rho)  \, T_1 = (1 + \rho) \, T_0  -  \pi(0,0) . \]
   Summing all equations from $k = 0$ to infinity yields
   \[ (1 + \rho) \, T_0  - A(1)   ~=~ 2 \rho \sum _{k=1}^{\infty}  \  T_k . \] 
   Besides,  since equation ~\eqref{diago} is still valid,  here for all $k \in \N$, with $ \pi$ in place of $\pi_K$, we first get by summation  over $k \in \N$
    \[ \sum _{ j = k } \pi(j,k) =   \frac {1}{ \rho}  \sum _{ 0 \le j < k } \pi(j,k) , \] 
    and then using  $\ \displaystyle{1 = \sum_{(j,k) \in \N^2}  \pi(j,k) =  \sum _{ j = k } \pi(j,k) + 2 \sum _{ 0 \le j < k } \pi(j,k)}  $\,, 
\[ T_0 =  \sum _{ k=0 } ^{ \infty}\pi(k,k) =   \frac {1 }{1+2 \rho}  \quad \text{ and } \quad  \sum _{k=1}^{\infty}  \  T_k =\sum _{0 \le j < k} \pi(j,k) =   \frac { \rho }{1+2 \rho} \, \cdot \] 
    Thus $A(1)$ is  determined and  given by
     \begin{equation} \label {A(1)} A(1) ~=~ 1- \rho. 
      \end{equation}

 In view of the next convergence result,  it is convenient to complete  Theorem \ref{kingman} with a uniqueness result which proof is contained in  the proof of Theorem \ref{cohen} that concludes this section.
  \begin {lemma} \label{uniquenessA}
   $A(y)$ is the unique analytic  function in the domain $|y| < 1+ 2 \rho $ that  satisfies (ii) of Theorem \ref{kingman} and
\[ A(1) ~=~ 1- \rho . \]      
 \end{lemma}

\medskip
The following result is a consequence of Theorem \ref{pi_K}, Theorem \ref{pi} and Lemma \ref{uniquenessA}. Here,   $\pi _K$ (for $K \in \N$) is  considered as  a measure on $\N ^2$.

\begin{corollary} For $\rho <1$, 
 $\pi _K$ converges weakly  to  $\pi$  as $K$ goes to infinity. 

\end{corollary}

\begin{proof}  Assuming that $\rho <1$, we prove that for each $(j,k) \in \N ^2$, $ \pi _K(j,k)$ converges to $ \pi (j,k)$ as $K$ tends to infinity. In view of Theorem \ref{pi_K} (i)  and Theorem \ref{pi}, it is enough to show that  for each  $l \in \N$, $ \lim _{K \to \infty} \pi _K(0,l)= \pi (0,l)$. Note indeed that since $g^{\ast m}(j) = 0 $ for $ j<m$, Theorem  \ref{pi_K} (i) can be written, for   $ 0 \le j < k \le K$,  as
\[\pi _K (j,k ) =     \sum _{l=k}^{ (2k-1) \wedge K} \pi _K(0,l) \cdot  \left (  g^{\ast (l-k+1)}(j)  -   g^{\ast (l-k+1)} (j+1) \right ) ,\]
\[ \text{and} \quad \pi _K(k,k) =  - \frac{1}{ \rho }   \sum _{l=k+1}^{(2k+1) \wedge  K} \pi _K(0,l)  \,  g ^{ \ast {(l-k)}}  (k+1)  \quad \text{ for} \quad  0 \le k \le K . \]
We now set $\pi _K(0,l) = 0$ for $l>K$ and define,  for $ K \in \N$, the  probability measure $q_K$ on $\N$  by 
\[ q_K(l) = \frac{\pi _K(0,l)}{A_K(1)} \qquad (l \in \N). \]
Using equation \eqref{A_K} with $x= 1$ (for which the roots of $p_x$ are $1$ and $1+2 \rho$) yields
\[ A_K (1+2 \rho) = \frac { (1+ \rho ) A_K(1) - 2 \rho ^2 \pi_K(K,K)}{(1- \rho) (1 + 2 \rho)}. \]
Together with equation \eqref {A_K(1)} and $\lim _{K \to \infty} \pi_K(K,K) = 0$ for $\rho <1$ (see Proposition \ref{asymptotics} in Section 2.2), this gives the following limits:
\[ \lim _{K \to \infty} A_K(1) = 1- \rho \qquad \text{and} \qquad   \lim _{K \to \infty} A_K(1+2 \rho) = \frac { (1+ \rho )}{ (1 + 2 \rho)}. \]
 In particular, one gets    $\, M  \stackrel{def}{=}  \sup _{K} A_K(1+2 \rho) < \infty \,$ and $\  \delta   \stackrel{def}{=} \inf _K     A_K(1) >0$. Then,
$ q_K(l) \le  (1+ 2 \rho ) ^{-l}   A_K(1+2 \rho)  / A_K(1) \le  \frac{M}{\delta} (1+ 2 \rho ) ^{-l}$ for $K,l \in \N$,  so that
   \[\lim _{L \to \infty} \sup _K \sum _{l=L}^ {\infty} q_K(l) =0.  \]
It then results from Prokhorov's theorem that the sequence  of probability measures $(q_K)$ is tight.

Now consider any weakly converging subsequence of $(q_K)$  and denote by $q$ its limit.  For simplicity, we abusively denote $q_K$ the generic term of this subsequence. Weak convergence implies that for $z \in \C$ with $|z| <1$,
\[  Q(z) ~\stackrel{def}{=}~ \sum _{l=0}^\infty q (l)\,  z^l  ~=~  \lim _{K \to \infty } \sum _{l=0}^\infty q_{K} (l)\,  z^l ~=   \lim _{K \to \infty }\frac{A_K(z)}{A_K(1)} , \]
and so, for $|z|<1$,  $ \lim  _{K \to \infty } A_K(z) = (1- \rho) \, Q(z) $.  Taking the limit $K \to \infty$ in equation \eqref{A_K}  and using again   $\lim _{K \to \infty} \pi_K(K,K) = 0$ then gives 
\[ \phi(y,z)   Q(y) - \phi(z,y) Q(z) = 0\] 
if   $|y| <1, |z|<1$ and $y,z$  are the roots of some polynomial $p_x$. Now $Q$ is analytic in the domain $|z| <1+2\rho$, since  Fatou's lemma gives 
\[ \sum _{l=0}^{\infty} q(l) (1+2 \rho ) ^l  \le  \varliminf _{K\to \infty}  \sum _{l=0}^{\infty} q_K(l) (1+2 \rho ) ^l  =\varliminf _{K\to \infty}  \frac{ A_K(1+2 \rho)}{A_K(1)}  = \frac { (1+ \rho )}{(1-\rho)  (1 + 2 \rho)}  < \infty .\]
The function $(1- \rho) Q$ then satisfies the two conditions of Lemma \ref{uniquenessA} that characterize $A$. So $(1- \rho) Q = A$, that is, $(1- \rho)^{-1} \pi$ is the only possible limit of a subsequence of $(p_K)$. By the tightness property, it results that $(p_K)$ converges to  $(1- \rho)^{-1} \pi$, or else, that $(\pi _K)$ converges to  $\pi$.

 \end{proof}

    \medskip
   Some notations and preliminary  observations are now required in order to formulate Cohen's result.
   First define
    \begin{multline*} \mathcal C = \{ (y,z) \in \C^2,  \ \exists \, x \in \C,  \  p_x \text{ has roots } y \text { and } z \} \\
    = \{ (y,z) \in \C^2,  \ \exists \, x \in \C, \,  y+z = 2( 1 + \rho ) x  \text{ and } yz = (1+2 \rho x) x \}  \\
    = \{ (y,z) \in \C^2,  \   2(1+ \rho )^2 yz =  (y+z) \left ( 1 + \rho + \rho (y+z) \right )  \} . 
    \end{multline*}
$\mathcal C$ is a Riemann surface which  is  invariant under the symmetry $(y,z) \mapsto (z,y)$.
Let define $a$ and $b$ by
 \[ a= \frac{1+ \rho}{2 \, (1+ \rho ^2)} \quad \text{ and } \quad b =   \frac{1}{2  \,  \sqrt {1+ \rho ^2}}  \,  \]
 and note that $0 <b<a$. Then $\mathcal C$ is  equivalently
characterized by the following equation 
\begin{equation} \label {C}    \frac{(2a-y-z) ^2}{a^2} - \frac{(y-z)^2}{b^2} =4 .  
\end{equation} 
$\mathcal C$ also has a parametric description, as 
 \begin{equation} \label{parametric} \mathcal C = \big \{ (a- a \cosh \theta + b \sinh \theta,  a-  a \cosh \theta - b \sinh \theta),\   \theta \in \C \big \}, 
 \end{equation}
in terms of  the hyperbolic functions  $\cosh$ and $\sinh$  with complex variable $\theta$ 
 \[  \cosh(\theta) = \frac {e^{\theta} + e^{- \theta}}{2} \qquad \text{ and }  \qquad  \sinh (\theta) = \frac {e^{\theta} - e^{-\theta}}{2}, \qquad \theta \in \C. \]
We have here used  the following equivalence,  for  $(u,v) \in \C^2$, 
 \[ u^2 - v^2 = 1 \quad  \Longleftrightarrow  \quad \exists \, \theta \in \C,  \ (u,v) = ( \cosh \theta, \sinh \theta) . \]

 Now  starting from any initial couple $(y,z) \in \mathcal C$, 
 one can built a \emph {chain} of couples in $ \mathcal C$. By this we mean that there is a unique sequence $(y^{(n)})_{n \in \Z}$ of complex numbers  that  satisfies the following two conditions.
 \begin{enumerate}
 \item  $y^{(0)} = y$, $y^{(1)} = z$ and
 \item for all  $n \in \Z$,  $y^{(n-1)}$ and $ y^{(n+1)} $  are the two (possibly equal) solutions $z$ of equation $  (y^{(n)} ,z) \in \mathcal C $, that is, of equation ~\eqref{C} with $y^{(n)}$ in place of $y$.
 \end{enumerate}
 This  results from the fact (see Remark \ref{degree_two}) that for given  $y$, equation $p_x(y) =0$  has two (possibly equal) solutions $x \in \C$.  
 
 Along such a chain,   for $n \in \Z$, $y^{(n)}$ and $y^{(n+1)}$  are the roots of some $p_{x^{(n)}}$. Or reversing roles,  $x^{(n-1)}$ and $x^{(n)}$  are the  roots of $ p_x(y^{(n)}) =0$.  Recall that 
 \[ p_x(y) ~=~ y^2 - 2(1+ \rho)\,  x  y + (1+2 \rho x) \, x ~=~ 2 \rho\,  x^2 - \Big  (  2(1+ \rho)\,  y  -1 \Big )  \, x +  y^2, \]
 so that the $x^{(n)}$'s and $y^{(n)}$'s  are related through  the following equations for   $n \in \Z$ : 
 \[ y^{(n)} +\,  y^{(n+1)} =   2(1+ \rho)\, x^{(n)} \quad \text{ and } \quad x^{(n-1)}+ x^{(n)} = ( 2 \rho ) ^{-1} \Big ( 2 (1+ \rho)\,   y^{(n)} -1 \Big )  . \]
 From this, one can  derive (by summing two consecutive equations of the first type, and next using the second equation) that $(y^{(n)})_{n \in \Z}$ satisfies the two-step recursion  
 \begin{equation} \label{recursion_y} y^{(n+1)} -2 \,  \frac{1+ \rho + \rho ^2}{\rho } y ^{(n)} + y ^{(n-1)} = - \frac{ 1+ \rho}{ \rho}, \qquad n \in \Z,  
 \end{equation}
 which is easily solved, noting that the polynomial $X^2 -2 \,  \big (\rho  ^{-1}  + 1+ \rho \big ) X +  1$ has roots  $(a+b)/(a-b)$ and $(a-b)/(a+ b)$. We get the following formulation of  Lemma 3 of \cite{Kingman-1}.

 \begin{lemma} \label{chain}  (Kingman)
 
  For any  $(y,z) \in \mathcal C$, the associate sequence  $(y^{(n)})_{n \in \Z}$ satisfies
  \[ y^{(n)} = a + \alpha(y,z) \left ( \frac{a+b}{a-b}\right ) ^n + \beta (y,z)  \left ( \frac{a-b}{a+b}\right )^n, \qquad n\in \Z, 
  \]
 where $\ \alpha(y,z) ,\, \beta (y,z)$ are  such that $\, \alpha(y,z) \, \beta (y,z) = (a^2-b^2)/4\, $ and given by
 \[ \alpha(y, z)  = \frac{a-b}{ 4ab} \Big (a( z-y) +b (z+y)-2ab  \Big ) ,  \quad     \beta(y,z)  = \frac{a+b}{4ab} \Big (a( y-z) +b (z+y)-2ab  \Big ). \]
  \end{lemma}
   
\medskip 
 
  \begin{remark} \label {chains}
For given $y$, there are only  two  (possibly equal)  chains with $y^{(0)} = y$, corresponding to  the two possible choices of $y^{(1)}$. Those chains are equal up to  symmetry $n \mapsto -n$.   
   \end{remark}

\medskip 
 It is easily proved using equation $2(1+ \rho )^2 yz =  (y+z) \left ( 1 + \rho + \rho (y+z) \right ) $ of $\mathcal C$,  that equation $\phi(y,z) = 0$ has exactly two solutions  $(y,z) \in \mathcal C$, given by 
\[  (u_0,u_1) ~ \stackrel{def}{=}~ \left (0, -\frac{1+ \rho}{ \rho} \right ) \ \qquad \text{and}  \qquad    (v_0,v_1) ~\stackrel{def}{=} ~ \left (\frac {2+\rho }{ \rho ^2},  \frac{1}{ \rho} \right ) .\]
Define $u=(u_n)_{n \in \Z}$ and  $v=(v_n)_{n \in \Z}$ as the chains  $(y^{(n)})_{n \in \Z}$ obtained  for $(y,z)$  respectively equal to $(u_0,u_1) $ and $(v_0,v_1) $. It results from the definition of $\mathcal C$ that    
\begin{align*}
 u_{-1}=~ u_0 =~ 0 , \qquad     u_{-2} =~ u_1 = -\frac{1+ \rho}{ \rho}, 
 \end{align*}
   and more generally $ u_n = u_{-(n+1)} $   for $ n \in \Z$,  
while using ~\eqref{recursion_y}  gives 
\[ v_0 =~ \frac {2+\rho }{ \rho ^2}, \qquad v_1 =~  \frac{1}{ \rho},  \qquad v_2 = 1, \qquad v_3 = 1+2 \rho.  \]

\medskip 
The expression of $A$ derived in \cite{cohen1995two} can now be formulated. 

  \begin{theorem} \label{cohen}  (Cohen) For $y \in \C$ with $|y| < (2+ \rho)/ \rho ^2$,
  \[ A(y) =  C \  \dfrac {\prod _{n=1}^{\infty } \left ( 1 - y/u_n \right) } {\prod _{n=0}^{\infty } \left ( 1 - y/v_{-n} \right) }, \]
where the constant  $C$ is  such that
\[ A(1)  = 1- \rho . \]
 \end{theorem}
\medskip
The remaining part of this section is devoted to an elementary proof of this theorem. But first, a few more notations and simple results are  needed.

  For $n \in \N$ and $y \in \C$, we define 
\begin{equation}  \label{Q_n} Q _n(y) = \phi \left (y^{(n)},y^{(n+1)} \right ) \,  \phi  \left(y^{(-n)},y^{(-n-1)}\right  )  ,    
\end{equation}
 where $(y^{(n)})_{n \in \Z}$ is any of the two  chains having $y^{(0)} = y$ (see Remark \ref{chains}). Since both are mutually symmetric,  $Q_n(y)$ is well-defined.

The following lemma is  crucial.

\begin{lemma} \label{zeroes_poles}
The two real-valued sequences $u$ and $v$ satisfy the following properties.

\smallskip
\noindent 1.  For $n \ge 1$, \ $  u_n  \le u_1\, $ and $ \,   v_{-n}  > v_0 \,$.


\smallskip
\noindent 2. The series $\sum _{n=1} ^{\infty} |u_n| ^{-1}$ and $\sum _{n=1} ^{\infty} v_{-n} ^{-1}$ converge.
 
\noindent 3. For  $n \ge 1$, the mapping $\, y \in \C  \longmapsto  y^{(n)}  y^{(-n)}$  is a degree $2$ polynomial  and has roots  $u_{ -n}$ and $u_n$. In other words,  for  $y \in \C$,
\[  y^{(1)}  y^{(-1)}=  \lambda _1 \,  y   \left ( 1- \frac{y}{u_1} \right )   \   \   \text{and}   \quad       y^{(n)}  y^{(-n)} =   \lambda _n \left ( 1- \frac{y}{u_{n}} \right ) \left ( 1- \frac{y}{u_{-n} }\right ),\, n \ge 2 \]
where for  $n \ge 1$, $ \lambda _n$ is a constant.

\noindent 4. For  $n \in \N$, $Q_n$ defined in ~\eqref{Q_n}  is a degree $2$  polynomial  and has  roots $u_{-n}$ and $v_{-n}$. In other words,  for some constants $\mu _n, \, n \in \N$ and,  for  $y \in \C$,
\[Q_0(y) =  \mu _0 \,  y   \left ( 1- \frac{y}{v_0} \right )  \quad \text{ and } \quad Q_1(y) =  \mu _1 \,  y  \left ( 1- \frac{y}{v _{-1}} \right ), \]

while  for $ n \ge 2$,
\[     Q_n(y) =  \mu _n \left ( 1- \frac{y}{u_{-n}} \right ) \left ( 1- \frac{y}{v_{-n}} \right ) . \]

\end{lemma} 

\begin{proof}
From   Lemma~\ref{chain}, one gets the following for any real-valued chain  $(y ^{(n)})$. \\
If   $\alpha (y, y^{(1)}) >0$, then, for any $n \in \Z$, 
\begin{align*}
   y ^{(n)} < y ^{(n+1)}  \Longleftrightarrow \left ( \frac{a+b}{a-b} \right ) ^{2n}  >~ \frac {a-b}{a+b} \, \frac{ \beta (y, y^{(1)})}{\alpha (y, y^{(1)})} 
\end{align*}
while if $\alpha (y, y^{(1)}) <0$, then, for any $n \in \Z$,
 \begin{align*}
   y ^{(n)} < y ^{(n+1)}  \Longleftrightarrow \left ( \frac{a+b}{a-b} \right ) ^{2n}  < ~\frac {a-b}{a+b} \, \frac{ \beta (y, y^{(1)})}{\alpha (y, y^{(1)})}.
\end{align*}
We recall that $ \alpha(y,y^{(1)} ) \beta(y,y^{(1)}) = (a^2-b^2)/4$ so that  $\alpha(y,y^{(1)} )$ and $\beta(y,y^{(1)} )$  have  the same sign and are non zero.  Hence,  if $ \alpha (y, y^{(1)}) >0$ (resp.~$ \alpha (y, y^{(1)}) <0$), the sequence $(y ^{(n)})_{n \in \Z}$  first decreases  (resp.~increases), up to some time $n$,  after which it  is nondecreasing  (resp.~nonincreasing).  Equality $y ^{(n)} = y ^{(n+1)}$ can moreover occur only at this first $n$ at which $y ^{(n)}$ is    minimum   (resp.~maximum). Point 1 of the lemma then simply results from relations
\begin{align*} u_{-1} & =~ u_0 =~ 0 ~>~ u_1 = ~- \frac{1+ \rho}{\rho}, \\
  v_2 & =1~<~ v_{1} = \frac{1}{\rho } \quad \text{and} \quad   v_2 ~<~ v_3 =  1+2 \rho . 
  \end{align*}
Point 2   also results from Lemma \ref{chain}, that shows that the modulus of any chain $(y^{(n)}) $ goes to infinity exponentially fast as $|n| \to + \infty $.   

As for the two last properties, it is easily proved inductively, from relations
\[   y^{(1)} =   2(1+ \rho)\, x^{(0)} -y^{(0)}  \quad \text{ and } \quad   y^{(-1)} =   2(1+ \rho)\, x^{(-1)} -y^{(0)}\] 
together with   \eqref {recursion_y}, that for  any  $y$ ($= y ^{(0)}$) and  $n \ge 1$,
\[  y^{(n)} =  \alpha _n + \beta _n  \,    y^{(0)}  + \gamma_n \, x^{(0)}  \quad \text{ and } \quad   y^{(-n)} =  \alpha _n + \beta _n    \, y^{(0)}  + \gamma_n \,  x^{(-1)}, \]
where  $\alpha _n, \beta _n$ and $\gamma _n$  $(n \ge 1)$ are constants.  It results that
\[ y^{(n)}  y^{(-n)} = \left ( \alpha _n + \beta _n  \,    y^{(0)}  \right  ) ^2 +  \,  \gamma_n   \Big (  \alpha _n + \beta _n  \,    y^{(0)}  \Big )  \big ( x^{(0)} +   x^{(-1)} \big)  +  \gamma_n  ^2 \,   x^{(0)}  x^{(-1)}  \]
and using $x^{(-1)}+ x^{(0)} = ( 2 \rho ) ^{-1} \Big ( 2 (1+ \rho)\,   y^{(0)} -1 \Big )$ and  $x^{(-1)} \, x^{(0)} = \left ( y^{(0)} \right )   ^2 /( 2 \rho )$, one gets that $y^{(n)}  y^{(-n)}$ is polynomial with degree 2 as function of $y^{(0)}$.  The argument is  the same for point $4$ of the lemma, since $\phi$ is a bivariate affine function.

Moreover, the roots of those polynomials are easily identified, using the following equivalences. For $n \ge 1$,
\[   y^{(n)}     y^{(-n)} = 0   ~\Longleftrightarrow ~     y^{(n)} =  u_0  \text{ or } y^{(-n)} = u_0   ~\Longleftrightarrow ~  y =  u_{-n}  \text{ or }  y= u_n , \]
while for $n \in \N$,
\begin{align*}  
  \phi \left (y^{(n)},  y^{(n+1)}\right )  \phi  \left(y^{(-n)},  y^{(-(n+1))}\right ) =0
 \end{align*}
means that either  $\left (y^{(n)},  y^{(n+1)} \right)$  or $ \left  (y^{(-n)},  y^{(-(n+1))}\right  ) $  is equal to  either of  the couples $ (u_0, u_1)$   or  $ (v_0, v_1) $, which is equivalent to asserting that   $  y =  u_{-n}$  or $  y=  v_{-n} $.  
\end{proof}

We are now ready for proving Theorem \ref{cohen}. 

\begin {proof}
It is first proved that 
\begin{equation}  \label{A_infinite}  A(y) \, \phi(y,z)  =  A(z) \, \phi(z,y) \  \   \text{for } (y,z) \in \mathcal C  \  \text{ with } \  |y| < 1+ 2 \rho, \,  |z| < 1+ 2 \rho .
    \end{equation}
Next,  a solution of  \eqref{A_infinite}  that is analytic in the open disk $D(0, 1+ 2 \rho )$ is exhibited. It is finally proved that such a solution is unique, up to a multiplicative constant. 

\medskip The first step  uses the parametrized form of $\mathcal C$ given in \eqref{parametric}, from which (ii) of Theorem \ref{kingman} can be reformulated as follows,
\begin{multline*}  A( a- a \cosh \theta + b \sinh \theta)\,  \phi (a- a \cosh \theta + b \sinh \theta,  a-  a \cosh \theta -b \sinh \theta) \\ = A( a- a \cosh \theta - b \sinh \theta) \,  \phi (a- a \cosh \theta - b \sinh \theta,  a-  a \cosh \theta + b \sinh \theta)  
\end{multline*} 
for all $\theta $ in some neighborhood of $0$ in $\C$. 
By analyticity of both sides  with respect to $\theta \in \C$, equality extends to any $\theta$ at which it makes sense. This yields \eqref {A_infinite} by (i) of Theorem \ref{kingman}.

\medskip
Now the idea behind the construction of a particular solution of \eqref{A_infinite} is the following. 
A formal, heuristic, solution to equation
\[A(y) \, \phi(y,z)  =  A(z) \, \phi(z,y) \quad \text{for  } (y,z) \in \mathcal C  \]
 is given by the following infinite product
\[   \prod _{n=0}^{ \infty} \left (\phi (y^{(n)}, y^{(n+1)}) \phi (y^{(-n)}, y^{(-n-1)})\right ) ^{-1}, \]
as function of $y $, where  $(y^{(n)})_{n \in \Z}$ is any of the two chains with $y^{(0)}= y$ (due to mutual symmetry of  the chains, this formal product does not depend on which  one is chosen). To check this,  first note that the  solutions  $z$ to  $(y,z) \in \mathcal C$ are given by $y^{(1)}$ and $y^{(-1)}$, which respectively  generate the shifted chains $ (y^{(n+1)})$  and $ (y^{(n-1)})$. Then for example,  the infinite product at $y^{(1)}$ differs from that at $y$ only by one factor, namely,   $\phi (y, y^{(1)})$  is changed for   $\phi ( y^{(1)},y)$. This shows that  the relation is  satisfied at $(y,y^{(1)})$.

Before caring about convergence of the  product, the first problem occurs that $0$ is a pole (with multiplicity $2$), which should not be the case for $A$.  But since equation~\eqref{A_infinite} is preserved by multiplying $A$ by any function that is constant along chains (so that its values are equal at $y$ and $z$ for any $(y,z) \in \C$), we can multiply - this is again heuristic - by the  infinite product  
\begin{align*}
\displaystyle{y \prod _{n=1}^{ \infty} y^{(n)}\,  y^{(-n)}}. 
\end{align*}
This formally removes the pole $0$ (since by Lemma~\ref{zeroes_poles},  $y^{(1)}\,  y^{(-1 )}$ has root $0$).

 Now normalizing all factors and using Lemma~\ref{zeroes_poles},   we get the following heuristic solution
 \begin{multline}  \label{product} \Pi(y) \stackrel{def}{=} \dfrac{  y \prod _{n=1}^{ \infty}  \lambda _n  ^{-1}  \, y^{(n)}  y^{(-n)}  } { \prod _{n=0}^{ \infty}  \, \mu _n ^{-1}   \,    \phi (y^{(n)}, y^{(n+1)})  \, \phi (y^{(-n)}, y^{(-n-1)})} \\
 =  \prod_{ n \ge 1}   \left ( 1- \frac{y}{u_n} \right )\prod _{n \ge 0 }  
 \left ( 1- \frac{y}{v_{-n}} \right )  ^{-1}, \end{multline}
where all infinite products converge, due to point 2 of Lemma~\ref{zeroes_poles}. 

 Of course, a rigorous proof must deal with finite truncations of this product. The shift from $y$ to $y^{(1)}$  then introduces edge effects ignored by  the above heuristics.  Yet, the relation will be satisfied  thanks to the exponential decay of   $(y^{(n)})$ at symmetric rate as $ n$ goes to $+ \infty$ and $-\infty$.

 \smallskip
 We now prove that the infinite product $\Pi$ in \eqref{product} satisfies 
 \[ \Pi(y) \,  \phi(y,z) =   \Pi(z)  \, \phi(z,y)  \]
for all $ (y,z) \in \mathcal C $   such that $   y,z \in \C \setminus \{ v_{-n}, n \in \N \}$.
     For $N \in \N$, denote by $\Pi_N$  the partial product
 \[ \Pi _N (y) \stackrel{def}{=} \dfrac{  y  \, \prod _{n=1}^{ N}  \lambda _n  ^{-1} \, y^{(n)}  y^{(-n)}  } { \prod _{n=0}^{ N }\,  \mu _n ^{-1}  \,  \phi (y^{(n)}, y^{(n+1)})\,  \phi (y^{(-n)}, y^{(-n-1)})}. \] 
 Due to analyticity (again using the complex variable $\theta$ instead of $(y,z) \in \mathcal C$), it is enough to consider $y,z $ at which all factors in $ \Pi _N (y)$ and $ \Pi _N (z)$ for $N \in \N$ are all non zero.
  Let $(y,z) \in \mathcal C$ be given so, and choose  $(y^{(n)})_{n \in \Z}$ as the chain such that $y^{(0)}= y$ and  $y^{(1)}=z$. Then
  \[  \Pi _N (y) \, \frac{ \phi(y,z)}{\phi(z,y)} =   \Pi _N (y) \, \frac{ \phi(y,y^{(1)})}{\phi(y^{(1)},y)}  =   \Pi _N (y^{(1)}) \,  \frac{y^{(-N)}}{y^{(N+1)}}   \,   \frac{ \phi(y^{(N+1)},y^{(N+2)})}{\phi(y^{(-N)},y^{(-N-1)})} . \]
  Moreover, from the definition \eqref{phi} of $\phi$,
  \[  \frac{y^{(-N)}}{y^{(N+1)}}   \,   \frac{ \phi(y^{(N+1)},y^{(N+2)})}{\phi(y^{(-N)},y^{(-N-1)})} =\frac { \rho -  y^{(N+2)} / y^{(N+1)}  - (1+ \rho)/ (\rho y^{(N+1)} )} {\rho -  y^{(-N-1)} / y^{(-N)}  - (1+ \rho)/ (\rho y^{(-N)} )} .\]
Now from  Lemma~\ref{chain}, the following limits holds
  \[ \lim _{ N \to +\infty}   y^{(N)}  \left ( \frac{a-b}{a+b} \right ) ^{N} =  \, \alpha (y,y ^{(1)})  \   \text { and }    \lim _{ N \to \infty}   y^{(-N)}  \left ( \frac{a-b}{a+b} \right ) ^{N} = \,  \beta (y,y ^{(1)}) ,\]
  so that (recall that  $\alpha(y,z )$ and $\beta(y,z )$ are always non zero)
   \[  \lim _{ |N| \to +\infty} y^{(N)} = + \infty  \quad \text { and } \quad     \lim _{ |N| \to +\infty} \frac {y^{(N+2)}}{y^{(N+1)} }  =  \lim _{ |N| \to +\infty} \frac { y^{(-N-1)} }{y^{(-N)} }  =  \frac{a+b}{a-b}.\]
   We get that 
   \[ 1 =   \lim _{ N \to +\infty}  \frac { \rho -  y^{(N+2)} / y^{(N+1)}  - (1+ \rho)/ (\rho y^{(N+1)} )} {\rho -  y^{(-N-1)} / y^{(-N)}  - (1+ \rho)/ (\rho y^{(-N)} )} =  \lim _{ N \to +\infty} \frac{ \Pi _N (y) }{  \Pi _N (z) } \, \frac{ \phi(y,z)}{\phi(z,y)},\]
   which yields  $\,  \Pi(y) \phi(y,z) =   \Pi(z) \phi(z,y)$, since  $ \Pi _N  $  goes to  $\Pi  $ as $N$ goes to infinity.

    \medskip
        It is now proved that $A$ is equal to $\Pi$. First note that,  from point 1 of Lemma \ref{zeroes_poles}, $v_0= (2+ \rho)/ \rho ^2$ is the pole of $\Pi$ with smallest modulus. Then, since $1+2 \rho < (2+ \rho)/ \rho ^2$ for all $\rho \in [0,1[$, it results that $\Pi$ is holomorphic in the open disk $D(0,1+2 \rho)$.  We thus know that   \eqref{A_infinite} is satisfied both for $A$ and for $\Pi$ in place of $A$, and  then get by taking ratios
    \begin{equation} \label {ratio} \frac{A(y)}{\Pi (y)} = \frac{A(z)}{\Pi (z)}
    \end{equation}
    for all $(y,z) \in \mathcal C$ such that $|y| < 1+ 2 \rho$,  $|z| < 1+ 2 \rho\, $ and $\, \Pi( y) \,  \Pi (z) \neq 0$.
         
     Consider now  the particular subset $ \mathcal E$  of  $  \mathcal C$  obtained by restricting $\theta$ to  $ i \, \R$  in description \eqref{parametric} of $ \mathcal C$.  Then, $ \mathcal E$ is the set of couples $(y, \overline y)  \in \C$ for $y$ on the ellipse $ \{ a-a \cos t + i \, b \sin t, t \in \R \}$. For $ \rho \in \, ]0,1[$,  we have $2b < 2a = (1+ \rho)/(1+ \rho ^2) < 1+ \rho < 1+ 2 \rho$, so that  $\mathcal E$  is contained in the open disk $D(0,1+2 \rho)$. This is also the case for the bounded open domain $E$ of the complex plane delimited by $\mathcal E$.  Note that $\mathcal E$  is also contained in the half plane $ \{y \in \C, \, \Re e (y) \ge 0 \}$, so that $\Pi$ does not vanish on $\mathcal E$.  Indeed, $\Pi$ only  has  real negative roots, of which  $-(1+ \rho)/\rho$ is the largest one. For $(y,z) = (y, \overline y)$ with $y \in \mathcal E$, equation \eqref{ratio} becomes 
     \begin{equation} \label{invariance} \frac{A(y)}{ \Pi (y) } = \frac{A (\overline y) }{ \Pi   (\overline y) } \qquad \text{ for } y \in \mathcal E.
     \end{equation}
      Now, $A/ \, \Pi$ is analytic in $D(0,1+ 2 \rho) \cap \{y \in \C, \, \Re e (y) > -(1+ \rho)/\rho \}$, so that $(A/ \, \Pi)(y)$ and $(A/ \,  \Pi) (\overline y)$ are  harmonic in this domain, and in particular,  harmonic over $E$ and  continuous over $E \cup \mathcal E$.
    
    By uniqueness of the  extension of a given  continuous function  on $\mathcal E$ into a continuous function on $E \cup \mathcal E$ that is harmonic over $E$, we derive that   \eqref{invariance} extends to all $y \in E$.  This means that $A/  \, \Pi$ is both holomorphic and antiholomorphic over $E$. Since $E$ is a  connected  open subset of  $\C$, this implies that $A/ \,  \Pi$ is constant over $E$. The proof is complete using relation~\eqref{A(1)}. 
            \end{proof}

            \begin{remark} \label{A(1/rho)}
            By analyticity, as used for equation \eqref{A_infinite}, the functional equation \eqref{A}  extends to all $(y,z) \in \mathcal C$ at which it makes sense.
 Noting  that $A(1/\rho) < + \infty$ since  $A$ has radius of convergence $ (2+ \rho)/\rho ^2 > 1/\rho $, the functional equation \eqref{A} at $x=1/(2 \rho) $ then relates $A(1/\rho)$ with $A(1)$  and yields
\begin{equation*}
	A(1/\rho)  = (2- \rho)A(1) = (2-\rho)(1-\rho).
\end{equation*}

            \end{remark}

   \section{The asymmetric  model.}   
      Theorems \ref{pi_K} (i)  and \ref{pi}   extend to the case where the queues have different service rates $\mu_1$ and $\mu_2$. In this asymmetric setting, one can also allow two different probabilities $p_i$,  $i =1,2$,   with $p_1 +p_2 =1$,  for choosing queue $i$  when  both queues have equal length.  Apart from these changes, the dynamics are the same as before.  The global arrival  rate is now denoted   $2 \lambda$, instead of $2 \rho$.

      Considering both  cases $K < \infty$ and $K= \infty$ at the same time, the stationary distribution of the queue-length process will  be simply denoted by $\pi$.   For  infinite $K$, it is assumed that  $2 \lambda < \mu _1 + \mu_2$, so that the process is ergodic.   In Figure~\ref{Q_asym}, the $Q$-matrix of the process is summarized through  a graphic showing  the  transitions and rates  for finite $K$.  For $K = \infty$,  the top and right borders of the square should simply be removed.

           \begin{figure}[ht]
		\centering
		\begin{tikzpicture}
			\draw[->, thick]
			 (0.6,0) -- (6,0) node[below] {$L_1$}
			;
			\draw[->, thick]
			 (0,0) -- (0,6) node[left] {$L_2$};
	
			\node at (-.2,-.2) {$0$};

			\draw[-, thick] (0,0) -- (5,5);
			
			\draw[ thick]
			 (0,5) node[left] {$K$}-- (5,5) ;
			\draw[ thick]
			 (5,0) node[below] {$K$}-- (5,4.4) ;

			\draw[->, thick]
				(1,1) -- ++	(0,.6)	node[left]	{\scriptsize{$2 \lambda p_2$}};
			\draw[->, thick]
				(1.,1.) -- ++	(-.6,0)	node[below]	{\scriptsize{$\mu _1$}};
			\draw[->, thick]
				(1.,1.) -- ++	(0,-.6)	node[right]	{\scriptsize{$\mu_2$}};
			\draw[->, thick]
				(1.,1.) -- ++	(.6,0)	node[above]	{\scriptsize{$2 \lambda p_1$}};

			\draw[->, thick]
				(0,0) -- ++	(0,.6)	node[left]	{\scriptsize{$2 \lambda p_2$}};
		
			\draw[->, thick]
				(0,0) -- ++	(.6,0)	node[above]	{\scriptsize{$2 \lambda p_1$}};

			
			\draw[->, thick]
				(5,5) -- ++	(-.6,0)	node[below]	{\scriptsize{$\mu_1$}};
			\draw[->, thick]
				(5,5) -- ++	(0,-.6)	node[right]	{\scriptsize{$\mu_2$}};
			
		
			\draw[->, thick]
				(0,5) -- ++	(0,-.6)	node[right]	{\scriptsize{$\mu_2$}};
			\draw[->, thick]
				(0,5) -- ++	(0.6,0)	node[above]	{\scriptsize{$2\lambda$}};

			
			\draw[->, thick]
				(1.65,3.2) -- ++	(-.6,0)	node[below]	{\scriptsize{$\mu _1$}};
			\draw[->, thick]
				(1.65,3.2) -- ++	(0,-.6)	node[right]	{\scriptsize{$\mu _2$}};
			\draw[->, thick]
				(1.65,3.2) -- ++	(0.6,0)	node[above]	{\scriptsize{$2\lambda$}};

			\draw[->, thick]
				(1.65,5) -- ++	(-.6,0)	node[below]	{\scriptsize{$\mu _1$}};
			\draw[->, thick]
				(1.65,5) -- ++	(0,-.6)	node[right]	{\scriptsize{$\mu _2$}};
			\draw[->, thick]
				(1.65,5) -- ++	(0.6,0)	node[above]	{\scriptsize{$2\lambda$}};

		
			\draw[->, thick]
				(0,3.2) -- ++	(0,-.6)	node[right]	{\scriptsize{$\mu _2$}};
			\draw[->, thick]
				(0,3.2) -- ++	(0.6,0)	node[above]	{\scriptsize{$2\lambda$}};

			\draw[->, thick]
				(5,1.65) -- ++	(-.6,0)	node[below]	{\scriptsize{$\mu _1$}};
			\draw[->, thick]
				(5,1.65) -- ++	(0,-.6)	node[right]	{\scriptsize{$\mu _2$}};
			\draw[->, thick]
				(5,1.65) -- ++	(0,0.6)	node[right]	{\scriptsize{$2\lambda$}};
		

			\draw[->, thick]
				(3.2,1.65) -- ++	(-.6,0)	node[below]	{\scriptsize{$\mu _1$}};
			\draw[->, thick]
				(3.2,1.65) -- ++	(0,-.6)	node[right]	{\scriptsize{$\mu _2$}};
			\draw[->, thick]
				(3.2,1.65) -- ++	(0,0.6)	node[right]	{\scriptsize{$2\lambda$}};			
			\draw[->, thick]
				(3.2,0) -- ++	(-.6,0)	node[below]	{\scriptsize{$\mu _1$}};
			\draw[->, thick]
				(3.2,0) -- ++	(0,0.6)	node[right]	{\scriptsize{$2\lambda$}};

		\end{tikzpicture}
		\caption{Transition rates of the asymmetric Markov process $(L_1,L_2)$}
 \label{Q_asym}
	\end{figure}
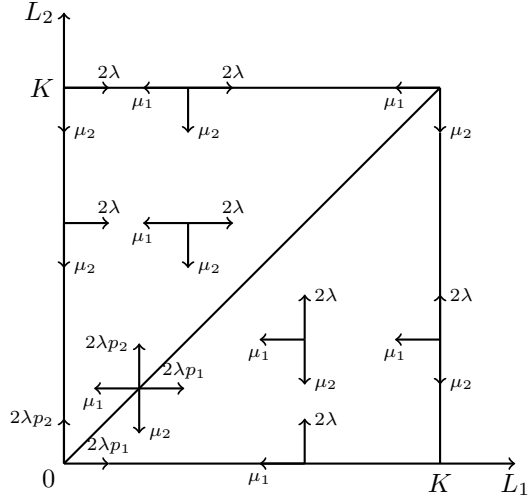
	
	 The stationary state $\pi$ is characterized by the following theorem, where $\pi(K,K)$ must be replaced by $0$ if $K = \infty$. Here,  $g_1$ and $ g_2$ are  defined on $ \N$ by
\[ g_1 (j) = - \frac { \mu _1}{\mu _2}\,  \frac { \xi _{1+} ^j -  \xi _{1-} ^j}{\xi _{1+}  -  \xi _{1-} } \quad \text{ and } \quad g_2 (j) = - \frac { \mu _2}{\mu _1} \, \frac { \xi _{2+} ^j -  \xi _{2-} ^j}{\xi _{2+}  -  \xi _{2-} } \quad (j \in \N), \]
where $ \xi _{1+} ,  \xi _{1-} $ are the roots of the polynomial $\mu _2 X^2 -(2 \lambda + \mu _1 + \mu _2) X + 2 \lambda$ and   $ \xi _{2+} ,  \xi _{2-} $ are those of  $\mu _1 X^2 -(2 \lambda + \mu _1 + \mu _2) X + 2 \lambda$. 

 \medskip
 Unfortunately here, for finite  $K$,  the  chain of equations that has led to the determination of the stationary blocking probability is no longer available. Indeed, one must now deal with  two generating functions
      \[ A_{1} (y) = \sum _{k=0} ^K \pi (k,0)\,  y ^k \quad \text { and }  \quad  A_{2} (y) = \sum _{k=0} ^K \pi (0,k) \,  y ^k, \]
      in place of $A$ or $A_K$.
      The corresponding relations that replace, for  $x \in \C$,  the functional equations  \eqref {A_K}  and  \eqref {A} 
       then    involve values  of both $A_1$ and $A_2$, respectively, at pairs of roots $y_1, z_1$ and $y_2, z_2$ of two  polynomials $p_{x,1}$  and $p_{x,2}$. As a result, one gets,  instead of chains of relations, a  branching   set of relations with  a degree-four regular tree structure.

      \begin{theorem}
      (i)  $\pi $  is determined by its values   $\pi(k,0) $ and $ \pi(0,k)$ for  $0 \le k < K+1$,  through the following expressions:  For $0 \le j <k <  K+1$,
\[   \pi (k,j) ~=~  \frac { \mu _2}{\mu _1} \,  \sum _{l=k}^{K} \pi  (l,0) \left ( g_1^{ \ast (l-k+1)}(j) -  g_1^{ \ast (l-k+1)}(j+1) \right ),  \]
\[  \pi (j,k) ~=~  \frac { \mu _1}{\mu _2} \,  \sum _{l=k}^{K} \pi  (0,l) \left ( g_2^{ \ast (l-k+1)}(j) -  g_2^{ \ast (l-k+1)}(j+1) \right ).  \]
 For $ \, 0 \le k  < K$,
\[  \pi (k,k) ~=~   -\frac{1}{2 \lambda}  \sum _{l=k+1}^{K} \left ( \mu _2 \,  \pi (l,0) \, g_1^{ \ast(l-k)}(k+1) +    \mu _1\,  \pi  (0,l) \,  g_2^{ \ast (l-k)}(k+1) \right )  \] 
and for $ K < \infty$,
\[  \pi (K,K) =   \frac{2 \lambda} {\mu _1+ \mu _2} \Big ( \frac { \mu _2}{\mu _1}    \pi (K,0)  \big (g_1 (K-1) - g_1(K)  \big )  +  \frac { \mu _1}{\mu _2}   \pi  (0,K)   \big(g_2 (K-1) - g_2(K)  \big)  \Big).   \]

\medskip
(ii) The sequences $(\pi(k,0), 0 \le k < K+1)$ and 
$(\pi(0,k), 0 \le k < K+1)$ are characterized, up to some (common) multiplicative constant, by the following relations holding for $x \in \C$  with $|x|$ sufficiently small: 
\begin{multline*}  \frac { \mu _2}{ y_{1}-z_{1}} \Big  ( ( y_{1}-x) A_1 ( y_{1}) -  ( z_{1}-x) A_1 ( z_{1})  \Big ) \\
= \mu _2\,  x ^K \, \pi (K,K) + \frac { \mu _2 + 2 p_1 \lambda x}{ 2\lambda } \left ( \mu _1 \frac {A_1( y_{1})-A_1(z_{1})} { y_{1}-z_{1}}  +   \mu _2  \frac {A_2( y_{2})-A_2(z_{2})} { y_{2}-z_{2}}   \right ) , 
\end{multline*}
and 
\begin{multline*}  \frac { \mu _1}{ y_{2}-z_{2}} \Big  ( ( y_{2}-x) A_2 ( y_{2}) -  ( z_{2}-x) A_2 ( z_{2})  \Big ) \\
= \mu _1\,  x^K \, \pi (K,K) + \frac { \mu _1 + 2 p_2 \lambda x}{ 2\lambda } \left (  \mu _1  \frac {A_1( y_{1})-A_1(z_{1})} { y_{1}-z_{1}}  +   \mu _2  \frac {A_2( y_{2})-A_2(z_{2})} { y_{2}-z_{2}}   \right ) , 
\end{multline*}
where for $x \in \C$, $y _{1} ,  z_{1} $ are the roots of the polynomial 
\[ p_{x,1}(Y) = \mu _2 Y^2 -(2 \lambda + \mu _1 + \mu _2) x Y + (\mu _1 +2 \lambda \, x)  x \]
and   $ y _{2} ,  z _{2} $ are the roots of the polynomial\[p_{x,2} (Y) = \mu _1 Y^2 -(2 \lambda + \mu _1 + \mu _2) x Y  + ( \mu _2 +2 \lambda \, x)  x. \]

\medskip 
(iii) The  characterization of $(\pi( k,0), 0 \le k < K+1)$ and 
$(\pi(0,k), 0 \le k < K+1)$  is complete with the following normalization  relation: 
\[ \mu_2  A_1(1)+ \mu _1 A_2(1) = \mu _1 + \mu _2 - 2 \lambda (1 - \pi (K,K)).\] 

\end{theorem}
\smallskip

 \begin{proof}  We only give a sketch of the proof.  Proving (i) is similar to the symmetric case, with  finite or infinite  capacity. In particular, the diagonal values $\pi(k,k)$ for  $k <K$ are derived from relations
 \[ \pi  (k,k) ~=~  \frac{1}{  2 \lambda   } \left (  \mu _1 \sum _{j=0}^{k} \pi  (k+1,j)  + \mu _2 \sum _{j=0}^{k} \pi  (j,k+1)  \right ) \qquad (0 \le k <K) \]
 that  here replace  \eqref{diago}. Those are obtained by summing  all balance equations in  squares $\{ 0, \dots , k \} ^2$. For finite $K$, the expression of $\pi(K,K)$ simply results from the balance equation at $(K,K)$.
 
 \smallskip
  As for (ii), one must use the balance equations at sites $(k-1,k)$ and $(k, k-1)$ for $1 \le k <K+1$. At  $(k-1,k)$, for example, we have
  \begin{multline*} \big(2 \lambda + \mu _1\1 _{ \{ k \ge 2 \}  } + \mu _2   \big )  \, \pi(k-1,k)  ~=~  2  \lambda\,  p_2 \, \pi(k-1,k-1)  
  ~+~ \mu _1  \, \pi(k,k) \\
  +~ 2 \lambda \,  \pi (k-2,k) \,  \1 _{ \{ k \ge 2 \}  }
  ~+~ \mu _2 \,  \pi(k-1,k+1) \,  \1 _{ \{k<K \}} . 
  \end{multline*}
 Multiplying this equation by $x ^k$ for $x \in \C$ and then summing up over $k$ leads to the second relation of (ii) (valid for  small $|x|$ only, if $K = \infty$, due to the use of Fubini's theorem). This long and tedious calculation is omitted here. We use the following lemma, that is easily derived from relation \eqref{fundamental_ter} of Section \ref{sec2}. Note that Lemma \ref{S_n}  could also be used there to recover the functional equation \eqref{A_K} from  Theorem \ref{pi_K}(ii).

 Finally, (iii) is derived in a similar way as relation \eqref{A(1)}. 
 
  \end{proof}

  \begin{lemma} \label{S_n}  For $i=1,2$, $n \in \N$  and $x \in \C$, define  $S^{(i)}_n(x)$   by
\[S^{(i)}_n(x)=  \sum _{k=0}^{n} x ^k  g _i^{ \ast {(n-k+1)}}  (k).\]

If $x \in \C$ is such that the complex roots $y_i$ and $z_i$ of the polynomial $p_{x,i}$  are not equal, then for $n \in \N$,
\[S^{(i)}_n(x) = - \frac { \mu _i}{\mu _{3-i}}  \,  x \,  \frac{y_i^n -z_i ^n}{y_i - z_i } . \] 
\end{lemma}
  
 \begin{remark}  
Reference  \cite{cohen1998analysis}  describes the poles and residues of $A_1(y)$ and $A_2(y)$, but a nice expression as the one in Theorem~\ref{cohen} is missing and constitutes a challenging issue. 
 \end{remark} 
 
\noindent  {\bf Acknowledgement.} The authors are grateful to the editor and referees for helpful discussions and, most particularly, to one referee for invaluable comments regarding the literature.


\begin{thebibliography}{10}

\bibitem{adan1994upper}
I.~Adan, G.J. van Houtum, and J.~van~der Wal.
\newblock Upper and lower bounds for the waiting time in the symmetric shortest
  queue system.
\newblock {\em Annals of Operations Research}, 48(2):197--217, 1994.




\bibitem{adan1990analysissym}
I.~Adan, J.~Wessels, and W.~H.~M. Zijm.
\newblock Analysis of the symmetric shortest queue problem.
\newblock {\em Comm. Statist. Stochastic Models}, 6(4):691--713, 1990.


\bibitem{adan1991analysisasym}
I.~Adan, J.~Wessels, and W.~H.~M. Zijm.
\newblock  Analysis of the asymmetric shortest queue problem.
\newblock  {\em  Queueing Systems}, 8(1):1--58, 1991.





\bibitem{adan1991analysis}
I.~Adan and J.~Wessels.
\newblock Analysis of the asymmetric shortest queue problem with threshold
  jockeying.
\newblock {\em Stochastic Models}, 7(4):615--627, 1991.


\bibitem{adan1993compensation}
I.~Adan, J.~Wessels, and W.~H.~M. Zijm.
\newblock  A compensation approach for two-dimensional markov processes.
\newblock {\em Advances in Applied Probability}, 25(04):783--817, 1993.




\bibitem{blanc1992power}
J.~P.~Blanc.
\newblock The power-series algorithm applied to the shortest-queue model.
\newblock {\em Operations Research}, 40(1):157--167, 1992.

 
  \bibitem{cohen1996symmetrical}
 J.~W.~Cohen.
 \newblock On the symmetrical shortest queue and the compensation approach.
 \newblock {\em Department of Operations Research, Statistics, and System Theory
   [BS]}, (R 9602):1--21, 1996.

\bibitem{cohen1995two}
J.~W.~Cohen.
\newblock Two-dimensional nearest-neighbour queueing models, a review and an
  example.
\newblock In Baccelli, F., Jean-Marie, A., Mitrani, I. (eds.) {\em Quantitative Methods in Parallel Systems}, pp.~141--152.
  Springer, Berlin, 1995.

\bibitem{cohen1998analysis}
J.~W.~Cohen.
\newblock Analysis of the asymmetrical shortest two-server queueing model.
\newblock {\em International Journal of Stochastic Analysis}, 11(2):115--162,
  1998.

\bibitem{Conolly-1}
B.~W.~Conolly.
\newblock The autostrada queueing problem.
\newblock {\em J. Appl. Probab.}, 21(2):394--403, 1984.

\bibitem{Dester2017Questa}
P.~S. Dester, C.~Fricker, and D.~Tibi.
\newblock Stationary analysis of the shortest queue problem.
\newblock {\em Queueing Systems}, 87(3):211--243, 2017.


\bibitem{eschenfeldt2015join}
P.~Eschenfeldt and D.~Gamarnik.
\newblock Join the shortest queue with many servers. The heavy traffic
  asymptotics.
\newblock {\em arXiv preprint arXiv:1502.00999}, 2015.

\bibitem{fayolle1979two}
G.~Fayolle and R.~Iasnogorodski.
\newblock Two coupled processors: the reduction to a Riemann-Hilbert problem.
\newblock {\em Zeitschrift f{\"u}r Wahrscheinlichkeitstheorie und verwandte
  Gebiete}, 47(3):325--351, 1979.
  
  \bibitem{fayolle2017random}
G.~Fayolle, R.~Iasnogorodski, and V.~Malyshev.
\newblock {\em Random walks in the quarter plane: algebraic methods, boundary
  value problems, applications to queueing systems and analytic combinatorics},
  vol.~40.
\newblock Springer International Publishing AG, Switzerland, 2017.


\bibitem{flatto1989longer}
L.~Flatto.
\newblock The longer queue model.
\newblock {\em Probability in the Engineering and Informational Sciences},
  3(04):537--559, 1989.

\bibitem{Flatto-1}
L.~Flatto and H.~P. McKean.
\newblock Two queues in parallel.
\newblock {\em Comm. Pure Appl. Math.}, 30(2):255--263, 1977.

\bibitem{foley2001join}
R.~D.~Foley and D.~R.~McDonald.
\newblock Join the shortest queue: Stability and exact asymptotics.
\newblock {\em Annals of Applied Probability}, 11(3):569--607, 2001.

\bibitem{fricker2014incentives}
C.~Fricker and N.~Gast.
\newblock Incentives and redistribution in homogeneous bike-sharing systems
  with stations of finite capacity.
\newblock {\em EURO Journal on Transportation and Logistics}, 1--31,
  2014.

\bibitem{guillemin2014stationary}
F.~Guillemin and A.~Simonian.
\newblock Stationary analysis of the shortest queue first service policy.
\newblock {\em Queueing Systems}, 77(4):393--426, 2014.

\bibitem{haight1958two}
F.~Haight.
\newblock Two queues in parallel.
\newblock {\em Biometrika}, 45(3-4):401--410, 1958.

\bibitem{Halfin-1}
S.~Halfin.
\newblock The shortest queue problem.
\newblock {\em J. Appl. Probab.}, 22(4):865--878, 1985.

\bibitem{Hooghiemstra1988Power}
G.~Hooghiemstra, M.~Keane, and S.~van~de Ree.
\newblock Power series for stationary distributions of coupled processor
  models.
\newblock {\em SIAM J. Appl. Math.}, 48(5):1159--1166, 1988.

\bibitem{katehakis2012successive}
M.~N. Katehakis and L.~C. Smit.
\newblock A successive lumping procedure for a class of markov chains.
\newblock {\em Probability in the Engineering and Informational Sciences},
  26(4):483--508, 2012.
  
  \bibitem{katehakis2015DES}
  M.~N. Katehakis, L.~C. Smit, and F.~M. Spieksma. 
  \newblock  DES and RES processes and their explicit solutions. 
 \newblock {\em Probability in the Engineering and Informational Sciences}, 
 29(2):191--217, 2015.


\bibitem{Kingman-1}
J.~Kingman.
\newblock Two similar queues in parallel.
\newblock {\em The Annals of Mathematical Statistics}, 32(4):1314--1323, 1961.

\bibitem{knessl1986two}
C.~Knessl, B.~Matkowsky, Z.~Schuss, and C.~Tier.
\newblock Two parallel queues with dynamic routing.
\newblock {\em IEEE transactions on communications}, 34(12):1170--1175, 1986.

\bibitem{knessl2011finite}
C.~Knessl and H.~Yao.
\newblock On the finite capacity shortest queue problem.
\newblock {\em Progress in Applied Mathematics}, 2(1):01--34, 2011.

\bibitem{Kurkova-1}
I.~A.~Kurkova and Y.~M.~Suhov.
\newblock Malyshev's theory and {JS}-queues. {A}symptotics of stationary
  probabilities.
\newblock {\em Ann. Appl. Probab.}, 13(4):1313--1354, 2003.


\bibitem{latouche}
G.~Latouche and V.~Ramaswami. 
\newblock  {\em  Introduction to Matrix Analytic Methods in Stochastic Modeling},
vol.~5.  
\newblock SIAM, Philadelphia, 1999.



\bibitem{li2007geometric}
H.~Li, M.~Miyazawa, and Y.~Q. Zhao.
\newblock Geometric decay in a qbd process with countable background states
  with applications to a join-the-shortest-queue model.
\newblock {\em Stochastic Models}, 23(3):413--438, 2007.

\bibitem{mitzenmacher1997analysis}
M.~Mitzenmacher.
\newblock On the analysis of randomized load balancing schemes.
\newblock In {\em Proceedings of the Ninth Annual ACM Symposium on Parallel
  Algorithms and Architectures}, pp.~292--301. ACM, 1997.

\bibitem{puhalskii2007large}
A.~A.~Puhalskii and A.~A.~Vladimirov.
\newblock A large deviation principle for join the shortest queue.
\newblock {\em Mathematics of Operations Research}, 32(3):700--710, 2007.

\bibitem{rao1987algorithmic}
B.~M.~Rao and M.~J.~M.~Posner.
\newblock Algorithmic and approximation analyses of the shorter queue model.
\newblock {\em Naval Research Logistics (NRL)}, 34(3):381--398, 1987.

\bibitem{ridder2005large}
A.~Ridder and A.~Schwartz.
\newblock Large deviations without principle: Join the shortest queue.
\newblock {\em Mathematical Methods of Operations Research}, 62(3):467--483,
  2005.

\bibitem{Tarabia-1}
A.~M.~K.~Tarabia.
\newblock Analysis of two queues in parallel with jockeying and restricted
  capacities.
\newblock {\em Appl. Math. Model.}, 32(5):802--810, 2008.

\bibitem{turner2000join}
S.~R.~E.~Turner.
\newblock A join the shorter queue model in heavy traffic.
\newblock {\em Journal of Applied Probability}, 37(01):212--223, 2000.

\bibitem{turner2000large}
S.~R.~E.~Turner.
\newblock Large deviations for join the shorter queue.
\newblock {\em Fields Institute Communications}, 28:95--108, 2000.

\bibitem{van1998bounds}
G.~J. Van~Houtum, W.~H.~M. Zijm, I.~J. B.~F. Adan, and J.~Wessels.
\newblock Bounds for performance characteristics: a systematic approach via
  cost structures.
\newblock {\em Stochastic Models}, 14(1-2):205--224, 1998.

\bibitem{van2009quasi}
J.~S.~H. Van~Leeuwaarden, M.~S. Squillante, and E.~M.~M. Winands.
\newblock Quasi-birth-and-death processes, lattice path counting, and
  hypergeometric functions.
\newblock {\em Journal of Applied Probability}, 46(2):507--520, 2009.



\bibitem{vvedenskaya1996queueing}
N.~Vvedenskaya, R.~Dobrushin, and F.~Karpelevich.
\newblock Queueing system with selection of the shortest of two queues: An
  asymptotic approach.
\newblock {\em Problemy Peredachi Informatsii}, 32(1):20--34, 1996.

\bibitem{weber1978optimal}
R.~R.~Weber.
\newblock On the optimal assignment of customers to parallel servers.
\newblock {\em Journal of Applied Probability}, 15(02):406--413, 1978.

\bibitem{whitt1986deciding}
W.~Whitt.
\newblock Deciding which queue to join: Some counterexamples.
\newblock {\em Operations research}, 34(1):55--62, 1986.

\bibitem{winston1977optimality}
W.~Winston.
\newblock Optimality of the shortest line discipline.
\newblock {\em Journal of Applied Probability}, 14(01):181--189, 1977.

\end{thebibliography}
\end{document}